\documentclass{amsart}
\usepackage[T1]{fontenc}
\usepackage{amsfonts}
\usepackage{mathrsfs}
\usepackage{mathtools}
\usepackage{amssymb}
\usepackage[colorlinks=true,linkcolor=blue,allcolors=blue]{hyperref}
\usepackage[noabbrev,capitalize,nameinlink]{cleveref}
\usepackage{amsmath,blkarray,booktabs,bigstrut}
\usepackage{graphicx}
\usepackage{dutchcal}
\usepackage{algorithm}
\usepackage{algpseudocode}
\usepackage{tikz}
\usepackage{tikz-cd}
\usepackage{graphicx}
\usepackage[normalem]{ulem}
\usepackage{enumitem}

\numberwithin{equation}{section}

\newcommand\R{{\mathbb R}}

\DeclareMathOperator{\supp}{supp}
\DeclareMathOperator{\spa}{span}

\DeclareMathOperator{\kernel}{Ker}
\DeclareMathOperator{\range}{Ran}
\DeclareMathOperator{\Dim}{Dim}

\DeclareMathOperator{\trace}{trace}

\newcommand\bof{{\bf f}}

\newcommand{\norm}[1]{\left\Vert#1\right\Vert}

\newcommand{\bR}{\mathbb R}
\newcommand{\bA}{{\bf A}}
\newcommand{\bS}{{\bf S}}
\newcommand{\bI}{{\bf Id}}
\newcommand{\rar}{\rightarrow}
\newcommand{\cH}{\mathcal{H}}

\theoremstyle{plain}
  \newtheorem{theorem}{Theorem}[section]
  \newtheorem{proposition}[theorem]{Proposition}
  \newtheorem{lemma}[theorem]{Lemma}
  \newtheorem{corollary}[theorem]{Corollary}

  \theoremstyle{remark}

\theoremstyle{definition}
  \newtheorem{definition}[theorem]{Definition}
  
  \newtheorem{example}[theorem]{Example}

\begin{document}
\include{psfig}
\title{Approximately Dual and Pseudo-Dual Probabilistic Frames}


\author{Dongwei Chen}
\address{Department of Mathematics, Colorado State University, Fort Collins, CO, US, 80523}
\email{dongwei.chen@colostate.edu}

\author{Emily J. King}
\address{Department of Mathematics, Colorado State University, Fort Collins, CO, US, 80523}
\email{emily.king@colostate.edu}

\author{Clayton Shonkwiler}
\address{Department of Mathematics, Colorado State University, Fort Collins, CO, US, 80523}
\email{clayton.shonkwiler@colostate.edu}
 
\begin{abstract}
This paper studies properties of dual probabilistic frames---in particular in relation to redundancy---and introduces both approximately dual probabilistic frames and pseudo-dual probabilistic frames. We show that the canonical dual probabilistic frame is the only dual frame of pushforward type of a probabilistic frame with zero redundancy.  Furthermore, we show that probabilistic frames with finite redundancy are atomic and finite. Approximately dual probabilistic frames generalize duality, with pseudo-duality being a further generalization. We introduce these concepts and prove certain structural results. In particular, every probabilistic frame has a discrete finite frame as an approximate dual.
\end{abstract}

\maketitle
\section{Introduction}

A sequence of vectors $\{{\bf f}_i\}_{i \in I}$ for $I$ an at-most countable index set in a separable Hilbert space $\mathcal{H}$ is called a \textit{frame} for $\mathcal{H}$ if  there exist \emph{frame bounds} $0<A \leq B < \infty$ such that for any ${\bf f} \in \mathcal{H}$, 
\begin{equation}\label{eqn:framedef}
    A \Vert {\bf f}  \Vert^2 \leq\sum_{i \in I} \vert \left\langle {\bf f},{\bf f}_i \right\rangle \vert ^2  \leq B\Vert {\bf f} \Vert^2.
\end{equation} 
Furthermore, $\{{\bf f}_i\}_{i \in I}$ is called a \emph{tight frame} if $A$ and $B$ may be chosen so that $A = B$ and a \emph{Parseval frame} if $A = B=1$. 

Frames were first introduced by Duffin and Schaeffer to analyze perturbations of Fourier series~\cite{duffin1952class}, and have since been applied in many areas, such as the Kadison–Singer problem~\cite{casazza2013kadison}, time-frequency analysis~\cite{grochenig2001foundations},  wavelet analysis~\cite{daubechies1992ten},  coding theory~\cite{strohmer2003grassmannian}, sampling theory~\cite{eldar2003sampling}, 
image processing~\cite{kutyniok2012multiscale},  compressed sensing~\cite{naidu2020construction}, and the design of filter banks~\cite{fickus2013finite}.   
Frames have been studied using tools from many mathematical areas, including operator theory \cite{christensen2017operator}, combinatorics \cite{bownik2015combinatorial}, matroid theory \cite{bernstein2020algebraic}, and algebraic geometry \cite{cahill2013algebraic}. In this paper, we study (probabilistic) frames from the perspective of probability theory and optimal transport \cite{ehler2012random, ehler2013probabilistic, wickman2017duality, wickman2014optimal}.  
For readers interested in more background on frame theory, see \cite{christensen2016introduction}.

Frames are often used to analyze signals and corresponding \emph{dual frames} are used to reconstruct signals. A sequence $\{{\bf g}_i\}_{i \in I}$  in $\mathcal{H}$ is called a \textit{dual frame} of the frame $\{{\bf f}_i\}_{i \in I}$ if
\begin{equation}\label{eqn:dualexpansion}
   {\bf f} = \sum_{i \in I}  \left\langle {\bf f},{\bf g}_i \right\rangle  {\bf f}_i   = \sum_{i \in I}  \left\langle {\bf f},{\bf f}_i \right\rangle  {\bf g}_i, \enskip\text{for any ${\bf f} \in \mathcal{H}$}.
\end{equation}
In this case, $\{{\bf g}_i\}_{i \in I}$ and $\{{\bf f}_i\}_{i\in I}$ are said to be a \textit{dual pair}. If the frame is not a Riesz basis, then there are infinitely many dual frames and thus infinitely many frame expansions for any $\bof$. One may think of this richness of frame expansions as being a measure of redundancy.

A particular choice of dual frame for $\{{\bf f}_i\}_{i\in I}$ is the \emph{canonical dual frame} $\{{\bf S}^{-1}{\bf f}_i\}_{i \in I}$ where ${\bf S}$ is the \textit{frame operator}:
\begin{equation*}
    {\bf S}: \mathcal{H} \rightarrow \mathcal{H}, \ \ \ {\bf S}({\bf f}) = \sum_{i \in I}  \left\langle {\bf f},{\bf f}_i \right\rangle  {\bf f}_i, \ \text{for any} \ {\bf f} \in \mathcal{H}.
\end{equation*}
Indeed, ${\bf S}: \mathcal{H} \rightarrow \mathcal{H}$ is a bounded, positive, and invertible operator. If $\bS = A \ \bI$ for some $A>0$, then the canonical dual is simply a scaling of the original frame, allowing for easy reconstruction. Note that  $\{{\bf f}_i\}_{i \in I}$  is tight with bound $A>0$ if and only if $\bS = A \ \bI$ and Parseval if $A=1$. In particular, when $\mathcal{H}$ is finite dimensional with dimension $n>0$, one obtains finite sums in all of the definitions above, and can choose $I = \{1, \hdots, N\}$ where $N \geq n$.

Suppose $\{{\bf y}_i\}_{i=1}^N$ is a finite frame for the Euclidean space $\mathbb{R}^n$ where $N \geq n$. Letting $\mu_y := \frac{1}{N} \sum\limits_{i=1}^N   \delta_{{\bf y}_i}$ be the uniform distribution on $\{{\bf y}_i\}_{i=1}^N$,  the frame definition in~\eqref{eqn:framedef} is equivalent to 
\begin{equation*}
    \frac{A}{N} \Vert {\bf x}  \Vert^2 \leq\int_{\mathbb{R}^n} \vert \left\langle {\bf x},{\bf y} \right\rangle \vert ^2 d\mu_y({\bf y})  \leq \frac{B}{N} \Vert {\bf x}  \Vert^2, \ \text{for any} \ {\bf x} \in \mathbb{R}^n,
\end{equation*}
which allows finite frames to be interpreted as discrete probabilistic measures. 
More generally, a probability measure $\mu$ on $\mathbb{R}^n$ is called a \emph{probabilistic frame} for $\mathbb{R}^n$ if there exist $0<A \leq B < \infty$ such that for any ${\bf x} \in \mathbb{R}^n$, 
\begin{equation*}
    A \Vert {\bf x}  \Vert^2 \leq\int_{\mathbb{R}^n} \vert \left\langle {\bf x},{\bf y} \right\rangle \vert ^2 d\mu({\bf y})  \leq B \Vert {\bf x}  \Vert^2.
\end{equation*}
Furthermore, $\mu$ is called a \emph{tight probabilistic frame} if $A = B$ and \emph{Parseval} if $A = B=1$. 

Probabilistic frames have finite second moments, which puts optimal transport and the $2$-Wasserstein distance in the realm of tools that can be used to study probabilistic frames. This is of interest even if one is only interested in finite frames, since $2$-Wasserstein distance can be used to quantify the distance between any two probabilistic frames, including atomic probabilistic frames coming from finite frames of different cardinalities.

Our goal here is to study duals of probabilistic frames and their generalizations.

\Cref{preliminary} gives some mathematical preliminaries on probabilistic frames and optimal transport.

\Cref{redundancy} concerns dual probabilistic frames and the redundancy of probabilistic frames. The definition of redundancy is motivated by Riesz bases, which are precisely the frames that yield unique frame expansions and, equivalently, are precisely the frames that have a so-called synthesis operator with a trivial kernel. In, e.g, \cite{hemmat2001generalized, balan2003deficits,hosseini2013structure,jakobsen2016density, speckbacher2019frames} redundancy is defined for discrete and continuous frames to be the dimension of the kernel of the synthesis operator. Applying this definition to probabilistic frames, we show in \cref{thm:finredfinsupp} that if a probabilistic frame has finite redundancy, it must be not only atomic but also finitely supported.

Given a probabilistic frame $\mu$, a measure $\nu$ with finite second moment is a \emph{dual probabilistic frame} for $\mu$ if there exists a transport coupling $\gamma \in \Gamma(\mu, \nu)$ such that 
\begin{equation*}
   \int_{\mathbb{R}^n \times \mathbb{R}^n} {\bf x}{\bf y}^t d\gamma({\bf x}, {\bf y}) = {\bf Id},
\end{equation*} 
where ${\bf y}^t$ is the transpose of the vector ${\bf y}$. We show in \cref{zeroRedundancy} that if a probabilistic frame has zero redundancy, the canonical dual is the only dual probabilistic frame of pushforward type, generalizing a well-known result for frames in Hilbert spaces. 

Characterizing all the dual probabilistic frames for a given probabilistic frame is important for signal reconstruction. This remains an open question, made more challenging by the fact that the set of dual probabilistic frames is not closed~\cite{chen2024general}. Therefore, in \cref{approximateDual},  we introduce approximately dual probabilistic frames, which can still produce perfect signal reconstruction: a measure $\nu$ with finite second moment is an \textit{approximately dual probabilistic frame} for the probabilistic frame $\mu$ if there exists $\gamma \in \Gamma(\mu, \nu)$ such that 
\begin{equation*}
   \left \| \int_{\mathbb{R}^n \times \mathbb{R}^n} {\bf x}{\bf y}^t d\gamma({\bf x}, {\bf y}) - {\bf Id} \right \|<1.
\end{equation*} 
We show that approximate duals can still yield perfect reconstruction.  
We also characterize all the approximately dual probabilistic frames that are pushforwards of the given probabilistic frame. 

In \cref{peturbation}, we study the approximately dual frames of perturbed probabilistic frames. We claim that if a probability measure is close to some probabilistic dual pair in an appropriate sense, then this probability measure is an approximately dual probabilistic frame. As a consequence, we show in \cref{thm:discrete approx dual} that every probabilistic frame has a \emph{finite, discrete} approximately dual probabilistic frame.

Finally, in \cref{pseudoDual} we consider a more general notion of dual frames, known as \emph{pseudo-dual probabilistic frames}.

\section{Preliminaries}\label{preliminary}
In this section, we give some necessary background on probabilistic frames and optimal transport. Throughout the paper,  we use ${\bf A}$ to denote a real matrix and $A$ a real number. Similarly,  ${\bf x}$ is used to denote a vector in $\mathbb{R}^n$ while $x$ is a point in the real line. In particular, ${\bf 0}$ is used to denote the zero vector in $\mathbb{R}^n$ while $0$ is the number zero. We use ${\bf Id}$ for the $n \times n$ identity matrix or the identity map for $\mathcal{H}$, ${\bf 0}_{n \times n}$ the $n \times n$ zero matrix, and  ${\bf A}^t$ for the transpose of a matrix ${\bf A}$.

\subsection{Probabilistic Frames}
Let $\mathcal{P}(\mathbb{R}^n)$ be the set of Borel probability measures on $\mathbb{R}^n$ and $\mathcal{P}_2(\mathbb{R}^n) \subset \mathcal{P}(\mathbb{R}^n)$  the set of Borel probability measures with finite second moments:
\begin{equation*}
   \mathcal{P}_2(\mathbb{R}^n) = \left\{ \mu \in \mathcal{P}(\mathbb{R}^n): M_2(\mu) = \int_{\mathbb{R}^n} \Vert {\bf x} \Vert ^2 d\mu({\bf x}) < + \infty \right\}.
\end{equation*}
Let $B_r({\bf x})$ be the open ball centered at ${\bf x}$ with radius $r>0$.  The \textit{support} of $\mu \in  \mathcal{P}(\mathbb{R}^n)$ is defined by 
$$\supp(\mu) = \{ {\bf x} \in \mathbb{R}^n: \text{for any} \ r>0, \,\mu(B_r({\bf x}))>0  \}.$$
Let $\mu \in  \mathcal{P}(\mathbb{R}^n) $ and $f: \mathbb{R}^n \rightarrow \mathbb{R}^m$ be a Borel measurable map where $n$ may be different from $m$, then $f_{\#} \mu \in \mathcal{P}(\mathbb{R}^m)$ is called the \textit{pushforward} of $\mu$ by  the map $f$,  and is defined as
$$f_{\#} \mu (E) := (\mu \circ f^{-1}) (E) = \mu \big (f^{-1}(E) \big ) \ \text{for any Borel set}  \ E \subset \mathbb{R}^m.$$ If $f$ is linear and represented by a matrix ${\bf A}$ with respect to some basis, then ${\bf A}_{\#} \mu$ is used to denote $f_{\#} \mu$. In particular, we use $(\mathbf{Id}, f)$ to denote the map from $\mathbb{R}^n$ to $\mathbb{R}^n \times \mathbb{R}^m$ via $\mathbf{x} \mapsto (\mathbf{x}, f(\mathbf{x}))$. In this case, $(\mathbf{Id}, f)_\#\mu$ is a probability measure on $\mathbb{R}^n \times \mathbb{R}^m$ that is supported on the graph of $f$.

Inspired by the insight that finite frames induce discrete probability measures, Martin Ehler and Kasso Okoudjou~\cite{ehler2012random, ehler2013probabilistic} proposed the concept of probabilistic frames as defined below. 

\begin{definition}
$\mu \in  \mathcal{P}(\mathbb{R}^n)$ is called a \emph{probabilistic frame} if there exist $0<A \leq B < \infty$ such that for any ${\bf x} \in \mathbb{R}^n$, 
\begin{equation*}
    A \Vert {\bf x}  \Vert^2 \leq\int_{\mathbb{R}^n} \vert \left\langle {\bf x},{\bf y} \right\rangle \vert ^2 d\mu({\bf y})  \leq B \Vert {\bf x}  \Vert^2.
\end{equation*}
A probabilistic frame $\mu$ is called \emph{tight} if $A = B$ may be chosen and \emph{Parseval} if $A = B=1$. Moreover, $\mu$ is a \emph{Bessel probability measure} if an upper bound (but possibly not a lower) holds. 
\end{definition}

Clearly if $\mu \in  \mathcal{P}_2(\mathbb{R}^n)$, $\mu$ is Bessel with bound $M_2(\mu) $. One can also define the \emph{frame operator} ${\bf S}_\mu$ for a probabilistic frame $\mu$  as the matrix
    \begin{equation*}
          {\bf S}_\mu := \int_{\mathbb{R}^n}  {\bf y} {\bf y}^t d \mu({\bf y}),
    \end{equation*}
and probabilistic frames can be characterized by frame operators. 

 \begin{proposition}[{\cite[Theorem 12.1]{ehler2013probabilistic}, \cite[Proposition 3.1]{maslouhi2019probabilistic}}] \label{TAcharacterization}
Let $\mu  \in \mathcal{P}(\mathbb{R}^n)$. 
\begin{enumerate}
    \item $\mu$ is a probabilistic frame if and only if the frame operator ${\bf S}_{\mu}$ is positive definite, which is also equivalent to  $\mu \in \mathcal{P}_2(\mathbb{R}^n)$ and $\spa\{\supp(\mu)\} = \mathbb{R}^n$.
   
     \item  $\mu$ is a tight probabilistic frame with bound $A > 0$ if and only if ${\bf S}_{\mu} = A \ {\bf Id}$. 
\end{enumerate}
\end{proposition}

\begin{corollary}\label{cor:full measure contains basis}
    If $\mu$ is a probabilistic frame for $\R^n$, then any Borel set $E \subseteq \R^n$ with $\mu(E) = 1$ contains a basis for $\R^n$.
\end{corollary}
\begin{proof}
    Since $\mu$ is a probabilistic frame, \cref{TAcharacterization} implies that $\spa\{\supp(\mu)\} = \R^n$, so $\supp(\mu)$ contains a basis ${\bf x}_1, \dots {\bf x}_n$ for $\R^n$. By taking $r>0$ sufficiently small, we can guarantee that, if ${\bf y}_i \in B_r({\bf x}_i)$ for all $i = 1, \dots , n$, then ${\bf y}_1, \dots , {\bf y}_n$ is a basis for $\R^n$. Since $\mu(B_r({\bf x}_i)) > 0$, we know that $E \cap B_r({\bf x}_i) \neq \emptyset$, so there is some choice of such a basis which is contained in $E$.
\end{proof}

\subsection{Optimal Transport and Dual Probabilistic Frames}

Since probabilistic frames are contained in $\mathcal{P}_2(\mathbb{R}^n)$, we can use the $2$-Wasserstein metric $ W_2(\mu,\nu)$ from optimal transport to quantify the distance between probabilistic frames $\mu$ and $\nu$:  
$$ W_2(\mu,\nu) := \left ( \underset{\gamma \in \Gamma(\mu,\nu)}{\operatorname{inf}}  \int_{\mathbb{R}^n \times \mathbb{R}^n}  \left \Vert {\bf x}-{\bf y}\right \Vert^2 \ d\gamma({\bf x},{\bf y}) \right )^{\frac{1}{2}}.$$
Here $\Gamma(\mu,\nu)$ is the set of transport couplings with marginals $\mu$ and $\nu$, given by 
$$ \Gamma(\mu,\nu) =  \left \{ \gamma \in \mathcal{P}(\mathbb{R}^n \times \mathbb{R}^n): {\pi_{{ x}}}_{\#} \gamma = \mu, \ {\pi_{{ y}}}_{\#} \gamma = \nu \right \},$$
where $\pi_{{ x}}$, $\pi_{{ y}}$ are the projections on the ${\bf x}$ and ${\bf y}$ coordinates, that is to say,  for any $({\bf x},{\bf y}) \in \mathbb{R}^n \times \mathbb{R}^n $, $\pi_{{ x}}({\bf x}, {\bf y}) = {\bf x}, \pi_{{ y}}({\bf x},{\bf y}) = {\bf y}$. Moreover, if $\gamma^* \in \Gamma(\mu, \nu)$ is an optimizer for $ W_2(\mu,\nu)$, then $\gamma^*$ is called an optimal transport coupling for $ W_2(\mu,\nu)$. See~\cite{figalli2021invitation} for more details on Wasserstein distance and optimal transport.

Like finite frames, one can reconstruct signals by \emph{dual probabilistic frames}. Dual probabilistic frames are also known as \emph{transport duals}, and it has been shown that every dual probabilistic frame is a probabilistic frame \cite{wickman2014optimal}. 

\begin{definition}\label{dualDefinition}
Suppose $\mu$ is a probabilistic frame. A measure $\nu \in \mathcal{P}_2(\mathbb{R}^n)$ is a \emph{dual probabilistic frame} of $\mu$ if there exists $\gamma \in \Gamma(\mu, \nu)$ such that 
\begin{equation*}
   \int_{\mathbb{R}^n \times \mathbb{R}^n} {\bf x}{\bf y}^t d\gamma({\bf x}, {\bf y}) = {\bf Id}.
\end{equation*}
\end{definition}

One of the dual probabilistic frames of a frame $\mu$ is the \emph{canonical dual probabilistic frame} ${{\bf S}_\mu^{-1}}_{\#}\mu$ (with bounds $\frac{1}{B}$ and $ \frac{1}{A}$) with coupling $\gamma:= {({\bf Id}, {\bf S}_\mu^{-1})}_{\#}\mu \in \Gamma(\mu, {{\bf S}_\mu^{-1}}_{\#}\mu)$. 
Then, we have the following reconstruction: for any ${\bf f} \in \mathbb{R}^n$, 
\begin{equation*}
   {\bf f}=   \int_{\mathbb{R}^n} \left\langle  {\bf f}, {\bf S}_\mu^{-1} {\bf x} \right\rangle {\bf x} \ d\mu({\bf x}) =\int_{\mathbb{R}^n} \left\langle  {\bf S}_\mu^{-1} {\bf f}, {\bf x} \right\rangle {\bf x} \ d\mu({\bf x}). 
\end{equation*}

We also need the following gluing lemma, a standard tool in optimal transport, to ``glue'' two transport couplings together. Similarly to how $\pi_x$ and $\pi_y$ were defined, let $\pi_{xy}$ and $\pi_{yz}$ be the projections on the corresponding coordinates, i.e.,  for any $({\bf x}, {\bf y}, {\bf z}) \in \mathbb{R}^n \times \mathbb{R}^n \times \mathbb{R}^n$, 
\begin{equation*}
    \pi_x({\bf x}, {\bf y}, {\bf z}) = {\bf x}, \ \pi_y({\bf x}, {\bf y}, {\bf z}) = {\bf y}, \ \pi_{xy}({\bf x}, {\bf y}, {\bf z})=({\bf x}, {\bf y}), \ \pi_{yz}({\bf x}, {\bf y}, {\bf z})=({\bf y}, {\bf z}).
\end{equation*}

\begin{lemma}[Gluing Lemma {\cite[p.~59]{figalli2021invitation}}]\label{gluinglemma} 
Let $\mu_1, \mu_2, \mu_3 \in \mathcal{P}_2(\mathbb{R}^n)$. Suppose $\gamma^{12} \in \Gamma(\mu_1, \mu_2)$ and $\gamma^{23} \in \Gamma(\mu_2, \mu_3)$ such that ${\pi_y}_{\#}\gamma^{12} = \mu_2 = {\pi_x}_{\#}\gamma^{23}$. Then there exists $\gamma^{123} \in \mathcal{P}(\mathbb{R}^n \times \mathbb{R}^n \times \mathbb{R}^n)$ such that ${\pi_{xy}}_{\#}\gamma^{123} = \gamma^{12}$ and ${\pi_{yz}}_{\#}\gamma^{123} = \gamma^{23}$. 
\end{lemma}

\section{Duals of Probabilistic Frames and Redundancy}\label{redundancy}
Every frame has at least one dual frame, the canonical dual, and thus every vector in the underlying Hilbert space has at least one expansion in the form of~\eqref{eqn:dualexpansion}. This expansion is unique if and only if the frame is additionally a so-called Riesz basis. We can characterize the space of expansions by considering the \emph{synthesis operator} $U^\ast$ of a frame, 
\[
U^\ast: \ell^2(I) \rar \cH, \quad U^\ast\left( (\omega_i)_{i \in I}\right) = \sum_{i \in I} \omega_i \bof_i,
\]
which by the definition of a frame is bounded with bound $\sqrt{B}$ for $B$ the upper frame bound. Furthermore (see, e.g.,~\cite{heil2011basis}), a frame is a Riesz basis if and only if the synthesis operator is injective and if and only if removing any single element results in a collection which is no longer a frame. Inspired by this, the authors in~\cite{balan2003deficits} defined the \emph{redundancy} $R(F)$ of a frame $F=\{{\bf f}_i\}_{i \in I}$ in a Hilbert space $\mathcal{H}$ as its excess, which is the greatest number of elements that can be removed yet still leave a frame. They further showed that the excess of a frame is the dimension of the kernel of the associated synthesis operator; that is to say, 
$$R(F) = \sup\{ |G|: G \subset F \ \text{and} \  F \backslash G \ \text{is a frame}\} = \Dim (\kernel U^*) =\Dim((\range U)^{\perp}),$$
where $|G|$ is the cardinality of the set $G$ and $U^*$ is the synthesis operator of $F$.  Furthermore, it has been shown that given a frame, all of the associated dual frames, approximately dual frames, and pseudo-dual frames have the same redundancy \cite{bakic2015excesses}. 

Under this definition, any Riesz basis has zero excess or redundancy. The authors in, e.g., \cite{hemmat2001generalized,hosseini2013structure,jakobsen2016density, speckbacher2019frames}
then generalized this to define the redundancy of \emph{continuous frames} in a similar way. For a fixed measure space $(X,\mu)$ and Hilbert space $\mathcal{H}$, a measurable map $\Psi: X \rightarrow \mathcal{H}$ is called a \emph{continuous frame} for $\mathcal{H}$ if there exist $0 < A \leq B$ such that for any ${\bf y} \in \mathcal{H}$, 
$$ A \Vert {\bf y} \Vert^2 \leq  \int_X  \vert \left\langle {\bf y},\Psi_x \right\rangle \vert ^2 d\mu(x) \leq B \Vert {\bf y} \Vert^2.$$
The associated \emph{analysis operator} is given by
$U: \mathcal{H} \rightarrow L_2(X, \mu)$, $[U{\bf y}](x) = \left\langle {\bf y},\Psi_x \right\rangle$, and the synthesis operator is
$$U^\ast:  L_2(X, \mu) \rightarrow \mathcal{H}, \quad U^\ast(\omega) = \int_X \omega(x)\Psi_x d\mu(x),$$
where the integral is defined weakly, and $U$ and $U^\ast$ are indeed adjoints of each other.  We can interpret a Riesz basis $\{\bof_i\}_{i=1}^\infty$ as being a continuous frame with respect to the counting measure on $X =\mathbb{N}$, where $\Dim (\kernel(U^*))=0$.
This justifies the definition of redundancy for the continuous frame $\Psi$:
$$R(\Psi):= \Dim(\kernel U^\ast) = \Dim((\range U)^{\perp}). 
$$ 
Note that the redundancy of a continuous frame need not be finite. The authors in \cite{hemmat2001generalized,hosseini2013structure,jakobsen2016density,speckbacher2019frames} also showed that if the redundancy of a continuous frame is finite, the associated measure space $(X,\mu)$ must be atomic. 

Probabilistic frames may be viewed as a type of continuous frame; however, in this section we highlight results specifically concerning probabilistic frames.
\begin{proposition}[{\cite[Proposition 1.12]{wickman2014optimal}}] \label{ctsframe_and_probframe}
    Suppose $\mu \in  \mathcal{P}(\mathbb{R}^n)$ is a probabilistic frame for $\mathbb{R}^n$ with frame bounds $0<A \leq B$. Then the identity map $\mathbf{Id}: \mathbb{R}^n \rightarrow \mathbb{R}^n$
    is a continuous frame for $\mathbb{R}^n$ with respect to $\mu$ and has the same bounds $0<A \leq B$.
\end{proposition}

Given a probabilistic frame $\mu$, the associated analysis operator $U: \mathbb{R}^n \rightarrow L^2(\mu, \mathbb{R}^n)$ and the synthesis operator $U^*:  L^2(\mu, \mathbb{R}^n) \rightarrow \mathbb{R}^n$ are given by 
$$U {\bf f} = \langle {\bf f}, \cdot \rangle \quad  \textrm{and} \quad U^* \omega = \int_{\mathbb{R}^n} {\bf x} \omega ({\bf x}) d\mu({\bf x}),$$
with $U$ and $U^\ast$ being adjoints and $\kernel U^*= (\range U)^{\perp}$. It follows from \cref{ctsframe_and_probframe} that these operators are the same as the operators for the corresponding continuous frame.  Thus, the redundancy $R(\mu)$ of a probabilistic frame $\mu$ is also defined as the dimension of the kernel of its synthesis operator.

This definition of redundancy also works for Bessel probability measures, since the associated analysis and synthesis operators are well-defined. The following example shows that when $\mu$ is finitely and uniformly supported, the redundancy of $\mu$ is the excess of the associated finite frame. 
\begin{example}
\normalfont
    Let $\mu = \delta_1$ be a probabilistic frame for the real line $\mathbb{R}$. Note that $f \in L^2(\mu, \mathbb{R})$ is really an equivalence class $[f]$, and $g \in [f]$ if and only if $f(x)=g(x)$ for $\mu$ almost all $x \in \mathbb{R}$. Since $\mu = \delta_1$, all the functions in $L^2(\mu, \mathbb{R})$ are totally determined by their function values at $1$. Then,  $g \in [f]$ if and only if $f(1)=g(1)$.  Clearly, the constant function $c=f(1)$ is also a representative of the equivalence class $[f]$. Therefore, $L^2(\mu, \mathbb{R}) \cong \mathbb{R}$.
    
   The synthesis operator $U^*: \mathbb{R} \rightarrow \mathbb{R}$ of $\mu$ is given by 
$$U^* \omega = \int_{\mathbb{R}} x \omega d\mu(x) = \omega, \quad \omega \in \mathbb{R}.$$
Then $\kernel U^* = \{\omega \in \mathbb{R}: U^*\omega =0\} = \{0\}$. Hence, $R(\mu) =\Dim(\kernel U^*) = 0$.

More generally, let $\mu = \frac{1}{N}\sum_{i=1}^N \delta_{{\bf y}_i}$ be a probabilistic frame for $\bR^n$ where $\{{\bf y}_i\}_{i=1}^N\subset\bR^n$.  Then $L^2(\mu, \mathbb{R}^n) \cong \ell^2(\{1, 2, \cdots, N\})$ and the synthesis operator of $\mu$ is the same as the synthesis operator of the finite frame $\{{\bf y}_i\}_{i=1}^N$ up to a constant $\frac{1}{N}$. Therefore, the kernels of these two synthesis operators are the same, and thus the redundancy of $\mu$ is the excess of the associated finite frame $\{{\bf y}_i\}_{i=1}^N$. 
\end{example} 

It is not necessarily true that the redundancy is weakly continuous; that is, if $\{\mu_n\}$ is a sequence of probabilistic frames converging weakly to a probabilistic frame $\mu$, then it does not necessarily hold that $R(\mu_n) \rightarrow R(\mu)$.
For example, let $\mu = \delta_1 $ and $\mu_n = \frac{1}{2}\delta_1 + \frac{1}{2} \delta_{1 -\frac{1}{n+1}}$ for each $n$, which are frames for the real line. Then $R(\mu_n) = 1$ for each $n$ and $R(\mu) = 0$. Note that $\mu_n$ converges to $\mu = \delta_1 $ weakly, and furthermore, $W_2(\mu_n, \mu) = \frac{1}{\sqrt{2}}\frac{1}{n+1} \rightarrow 0$, but $R(\mu_n)$ does not converge to $R(\mu)$. Therefore, if two probabilistic frames are close to each other, their redundancies need not be close to each other.

Balan et al.~\cite{balan2003deficits} showed that given a frame $\{{\bf f}_i\}_{i \in I}$ with frame operator $\bf S$ in a Hilbert space $\mathcal{H}$, the excess of $\{{\bf f}_i\}_{i \in I}$ is given by
$$R(\{{\bf f}_i\}) = \sum_{i \in I} \left( 1- \langle {\bf f}_i, {\bf S}^{-1}{\bf f}_i \rangle\right).$$
We give a similar expression for the redundancy of probabilistic frames.
\begin{lemma}
    Suppose $\mu$ is a probabilistic frame and $\{w_i\}_{i \in I}$ an orthonormal basis for $L^2(\mu, \mathbb{R}^n)$. Then the redundancy $R(\mu)$ for $\mu$ is given by
    \begin{equation*}
        R(\mu) = \sum_{i \in I} \left(1- \langle {\bf w}_i, {\bf S}_\mu^{-1}{\bf w}_i \rangle\right),
    \end{equation*}
    where  ${\bf w}_i = \int_{\mathbb{R}^n} {\bf x} w_i({\bf x})d\mu({\bf x})$ for each $i \in I$. 
\end{lemma}
\begin{proof}
    Let $U: \mathbb{R}^n \rightarrow L^2(\mu, \mathbb{R}^n)$ and $U^*:  L^2(\mu, \mathbb{R}^n) \rightarrow \mathbb{R}^n$ be the analysis and synthesis operators for $\mu$. If $P$ is the orthogonal projection onto $\kernel U^*$, then $P = I - U{\bf S}_\mu^{-1}U^*$.  By definition, 
    \begin{equation*}
        \begin{split}
            R(\mu) &= \Dim(\kernel U^*) = \trace (P) = \sum_{i \in I} \langle w_i, Pw_i \rangle_{L^2(\mu)}\\ 
            &= \sum_{i \in I} \left(1 - \langle w_i, U{\bf S}_\mu^{-1}U^*w_i \rangle_{L^2(\mu)} \right) 
            =\sum_{i \in I} \left(1- \langle {\bf w}_i, {\bf S}_\mu^{-1}{\bf w}_i \rangle\right),
        \end{split}
    \end{equation*}
 where ${\bf w}_i = U^*w_i = \int_{\mathbb{R}^n} {\bf x} w_i({\bf x})d\mu({\bf x})$ for each $i \in I$. 
\end{proof}

Two frames $\{{{\bf f}_i}\}_{i \in I}$ and $\{{{\bf g}_i}\}_{i \in I}$ in a Hilbert space $\mathcal{H}$ are said to be equivalent if there exists an invertible bounded linear operator $L$ on $\mathcal{H}$ such that ${\bf g}_i = L{\bf f}_{i}$, for each $i \in I$. Baki{\'c} and Beri{\'c}~\cite{bakic2015excesses} showed that if $L$ is only bounded, the excess of $\{{L{\bf f}_i}\}_{i \in I}$ is not smaller than that of $\{{{\bf f}_i}\}_{i \in I}$, i.e., $R(\{{{\bf f}_i}\}_{i \in I}) \leq  R(\{{L{\bf f}_i}\}_{i \in I})$, and thus if two frames in $\mathcal{H}$ are equivalent, they have the same excess. We show a similar result for equivalent probabilistic frames. 

Two probabilistic frames $\mu$ and $\nu$ are \emph{equivalent} if there exists an invertible $n \times n$ matrix ${\bf A}$ such that $\nu = {\bf A}_\#\mu$. The assumption that $\mathbf{A}$ is invertible is essential: if ${\bf A}_\#\mu$ is a probabilistic frame, then ${\bf A}$ must be invertible. To see this, observe that ${\bf A}_\#\mu$ being a probabilistic frame implies, by \cref{TAcharacterization}, that $\spa(\supp({\bf A}_\#\mu)) = \mathbb{R}^n$. Since $\supp({\bf A}_\#\mu) \subset \range({\bf A})$ and $\range({\bf A})$ is a subspace, then $\spa(\range({\bf A})) = \mathbb{R}^n = \range({\bf A})$. Thus, ${\bf A}$ has full rank and hence is invertible.

\begin{lemma}\label{lem:equivalent implies same redundancy}
Suppose $\mu$ and $\nu$ are equivalent probabilistic frames. Then $\mu$ and $\nu$ have the same redundancy. 
\end{lemma}
\begin{proof}
Since $\mu$ and $\nu$ are equivalent, there exists an invertible $n \times n$ matrix ${\bf A}$ such that $\nu = {\bf A}_\#\mu$. We claim that $L^2(\mu, \mathbb{R}^n)$ is isometrically isomorphic to $L^2(\nu, \mathbb{R}^n)$. We define the isometric isometry  $T: L^2(\nu, \mathbb{R}^n) \rightarrow L^2(\mu, \mathbb{R}^n)$ as
\[
Tf = f \circ \mathbf{A}, \quad f \in L^2(\nu, \mathbb{R}^n). 
\]
Note that $T$ is well-defined and preserves norms, since 
\[
\|Tf\|_{L^2(\mu, \mathbb{R}^n)}^2 = \int_{\mathbb{R}^n} |f(\mathbf{A}{\bf x})|^2  d\mu({\bf x}) =  \int_{\mathbb{R}^n} |f({\bf y})|^2 d\nu({\bf y}) = \|f\|_{L^2(\nu, \mathbb{R}^n)}^2 < \infty.
\]
$T$ is clearly linear and preserves the inner product: for any $f, g \in L^2(\nu, \mathbb{R}^n)$,
\[
\langle Tf, Tg \rangle_{L^2(\mu, \mathbb{R}^n)} = \int_{\mathbb{R}^n} f(\mathbf{Ax}) g(\mathbf{Ax})  d\mu({\bf x}) = \int_{\mathbb{R}^n} f({\bf y}) g({\bf y}) d\nu({\bf y}) = \langle f, g \rangle_{L^2(\nu, {\mathbb{R}^n})}.
\]
Furthermore, $T$ is bijective. If $Tf = 0$, then $\|Tf\|_{L^2(\mu, \mathbb{R}^n)} = \|f\|_{L^2(\nu, \mathbb{R}^n)} = 0$ and thus $f = 0$ in $L^2(\nu, \mathbb{R}^n)$. Hence, $T$ is injective. For the surjective part, for any $g \in L^2(\mu, \mathbb{R}^n)$, define $f = g \circ \mathbf{A}^{-1}$. Then $Tf = f \circ \mathbf{A} = g \circ \mathbf{A}^{-1} \circ \mathbf{A} = g
$ and $f \in L^2(\nu, \mathbb{R}^n)$ since $T$ is an isometry. Therefore, $T: L^2(\nu, \mathbb{R}^n) \rightarrow L^2(\mu, \mathbb{R}^n)$ is an isometric isomorphism.

Now let $U_\mu$ and $U_{\nu}$ be the analysis operators for $\mu$ and $\nu$. Since $\mu$ is a probabilistic frame, the associated analysis operator $U_\mu: \mathbb{R}^n \rightarrow L^2(\mu, \mathbb{R}^n)$ is bounded below. Therefore, $U_\mu$ is injective and its range $\range U_\mu$ is closed. By the Rank–Nullity theorem, $\Dim(\range U_\mu) = n$. Similarly, we have $\Dim(\range U_\nu) = n$. Thus, $\range U_\mu$ is isomorphic to $\range U_\nu$ (and isomorphic to $\mathbb{R}^n$). Therefore, $\Dim((\range U_\mu)^\perp) = \Dim((\range U_\nu)^\perp)$, that is to say, $R(\mu) = R(\nu)$. Hence, $\mu$ and $\nu$ have the same redundancy.
\end{proof}

Conversely, probabilistic frames can be classified by redundancy. Let $\mathcal{P}_{++} \subset \mathcal{P}_2(\mathbb{R}^n)$ be the set of probabilistic frames and define a relation $\sim$ on $\mathcal{P}_{++}$ as follows: $\mu \sim \nu$ if and only if $R(\mu) = R(\nu)$. Clearly, the relation $\sim$ is an equivalence relation on $\mathcal{P}_{++}$, and thus probabilistic frames can be classified by the value of their redundancy. 
Note, however, that if two probabilistic frames have the same redundancy, i.e., they are in the same equivalence class under the relation $\sim$,  these two frames are not necessarily equivalent to each other. For example, let $\mu = \frac{1}{2}\delta_{1} +\frac{1}{2}\delta_{2}$ and $\nu = \frac{1}{2}\delta_{2} +\frac{1}{2}\delta_{5}$. Then $\mu$ and $\nu$ are probabilistic frames for the real line and $R(\mu) = R(\nu)=1$. However, there is no invertible linear transformation $A$ so that $A_\#\mu=\nu$, since the matrix for $A$ is a $1 \times 1$ matrix $\begin{bmatrix} a \end{bmatrix}$ for some $a \neq 0$ and hence $A_\#\mu = \frac{1}{2}\delta_{a} +\frac{1}{2}\delta_{2a}$, which cannot equal $\nu = \frac{1}{2}\delta_{2} +\frac{1}{2}\delta_{5}$. 

It is known that if a continuous frame has finite redundancy, the associated measure space $(X, \mu)$ must be atomic \cite{speckbacher2019frames}. See also related results in~\cite[Theorem 2]{hosseini2013structure},
\cite[Theorem 2.2]{hemmat2001generalized}, and~\cite[Proposition~3.3]{jakobsen2016density}.

\begin{proposition}[{\cite[Proposition 9]{speckbacher2019frames}}]\label{prop:fd RKHS atomic}
    If a continuous frame with measure space $(X, \mu)$ has finite redundancy, then $(X, \mu)$ is atomic.
\end{proposition}

Recall that a measurable set $A \in \mathcal{A}$ is called an atom of a measure $\mu$ on the measurable space $(X, \mathcal{A})$ if
$\mu(A) > 0$ and for any $B \subseteq A$ with $B \in \mathcal{A}$, 
$\mu(B) = 0$ or $\mu(B) = \mu(A)$.
$\mu$ is called atomic if every measurable set of positive measure contains an atom, and the measure space $(X, \mathcal{A}, \mu)$ is atomic if $\mu$ is atomic. A discrete measure is a countable sum of weighted Dirac measures, which is a special case of atomic measures. 
Not every atomic measure is discrete since it may have uncountably many atoms. However, if the measure $\mu$ is $\sigma$-finite on $(X, \mathcal{A})$, then $\mu$ has at most countably many atoms that are necessarily not dense, which in turn implies $\mu$ is also discrete \cite[p.~31]{chung2000course}. 

Combining \cref{ctsframe_and_probframe,prop:fd RKHS atomic}, we see that a probabilistic frame with finite redundancy is atomic. Since probabilistic frames are Borel probability measures on $\mathbb{R}^n$, they are $\sigma$-finite. Therefore, a probabilistic frame is atomic if and only if it is discrete. Then, probabilistic frames with finite redundancy must be discrete.

In fact, even more is true for probabilistic frames with finite redundancy: they must have finite support.

\begin{theorem}\label{thm:finredfinsupp}
A probabilistic frame $\mu$ for $\mathbb{R}^n$ has finite redundancy if and only if  $\mu$ is finitely supported.
\end{theorem}
\begin{proof}
Clearly,  if $\mu$ is finitely supported, then the dimension of the kernel of the associated synthesis operator is no greater than the cardinality of its support, which means the redundancy is finite. Conversely, if a probabilistic frame $\mu$ has finite redundancy, then by \cref{ctsframe_and_probframe}, the associated continuous frame has finite redundancy. By \cref{prop:fd RKHS atomic},
$\mu$ must be atomic, that is, $\mu = \sum_{k\in I} p_k \delta_{{\bf x}_k}$ where $\{{\bf x}_k\}_{k\in I} \subset \mathbb{R}^n$, $p_k>0, \sum_{k\in I} p_k = 1$, and $I$ is a countable index set. First, we claim $I$ is finite. Note that the associated synthesis operator
$U^*$ maps $L^2(\mu, \mathbb{R}^n)$ to $\mathbb{R}^n$. If $I$ is infinite, then $L^2(\mu, \mathbb{R}^n)$ is infinite-dimensional. The Rank–Nullity theorem implies that $\Dim(\kernel U^*)$ is also infinite-dimensional, and hence $R(\mu) = \Dim(\kernel U^*) = \infty$. Thus, if $R(\mu) < \infty$, the index set $I$ must be finite. Therefore, $\mu$ is finitely supported. 
\end{proof}

Given a probabilistic frame $\mu$, Clare Wickman and Kasso Okoudjou ~\cite{wickman2017duality} have shown that for each map $h: \mathbb{R}^n \rightarrow \mathbb{R}^n$ such that $h_\# \mu \in \mathcal{P}_2(\mathbb{R}^n)$, we can define a function $T:\mathbb{R}^n \rightarrow \mathbb{R}^n$ by
      \begin{equation*}
          T({\bf x}) := {{\bf S}_\mu^{-1}}{\bf x} + h({\bf x}) - \int_{\mathbb{R}^n} \langle {\bf S}_\mu^{-1}{\bf x}, {\bf y} \rangle h({\bf y}) d\mu({\bf y})
      \end{equation*}
so that $T_\# \mu$ is a dual probabilistic frame of $\mu$. We call a probabilistic frame of the form $T_\#\mu$ a frame of \emph{pushforward type}.

In the following, we give a characterization of dual frames of pushforward type. We also show that if the redundancy of a probabilistic frame is zero, the canonical dual is the only dual frame of pushforward type.

\begin{lemma}\label{dualChar}
    Suppose $\mu$ is a probabilistic frame and $T_\# \mu$ is a dual frame of $\mu$ where $T: \mathbb{R}^n \rightarrow \mathbb{R}^n$. Then $T$ must be of the form
    \begin{equation*}
        T({\bf x}) = {\bf S}_\mu^{-1}{\bf x} + h({\bf x}), 
    \end{equation*}
    where $h: \mathbb{R}^n \rightarrow \mathbb{R}^n$ is such that $h_\#\mu \in \mathcal{P}_2(\mathbb{R}^n)$ and $ \int_{\mathbb{R}^n } {\bf y} h({\bf y})^t d\mu({\bf y}) = {\bf 0}_{n \times n}$. 
\end{lemma}
This may be viewed as a corollary of a result concerning pseudo-dual probabilistic frames by taking ${\bf A =Id} $.  See~\cref{cor:psuedodualpush} for the proof.

\begin{proposition}\label{zeroRedundancy}
    Let $\mu$ be a probabilistic frame with zero redundancy and $T_\# \mu$ a dual probabilistic frame for $\mu$. Then, for $\mu$-almost all ${\bf x} \in \mathbb{R}^n$,  $T({\bf x}) = {\bf S}_\mu^{-1}{\bf x}.$
\end{proposition}
\begin{proof}
    Since $T_\# \mu$ is a dual frame for $\mu$, \cref{dualChar} implies that there exists $h: \mathbb{R}^n \rightarrow \mathbb{R}^n$ with $h_\#\mu \in \mathcal{P}_2(\mathbb{R}^n)$ and
    \begin{equation*}
        \int_{\mathbb{R}^n } {\bf z} h({\bf z})^t d\mu({\bf z}) = {\bf 0}_{n \times n},
    \end{equation*}
    so that $T({\bf x}) = {\bf S}_\mu^{-1}({\bf x}) + h({\bf x})$ for any ${\bf x} \in \R^n$.
Thus, for any ${\bf y} \in \mathbb{R}^n$, 
\begin{equation*}
        \int_{\mathbb{R}^n} {\bf z} \langle h({\bf z}), {\bf y} \rangle d\mu({\bf z}) = {\bf 0}.
\end{equation*}
This means that for every ${\bf y} \in \mathbb{R}^n$, $\langle h(\cdot), {\bf y} \rangle \in \kernel U^*$ where $U^*: L^2(\mu, \mathbb{R}^n) \rightarrow \mathbb{R}^n$ is the synthesis operator of $\mu$. However, $0=R(\mu)=\Dim(\kernel U^*)$, so $\kernel U^* = \{ 0\}$ where 0 should be seen as the zero function in $L^2(\mu, \mathbb{R}^n)$. Hence, for all  ${\bf y} \in \mathbb{R}^n$, the function $\langle h(\cdot), {\bf y} \rangle $ is the zero function in $L^2(\mu, \mathbb{R}^n)$, i.e., $\langle h(\cdot), \mathbf y \rangle = 0$ $\mu$-almost everywhere for each $\mathbf y \in \mathbb{R}^n$. Then $h({\bf x}) = \mathbf{0}$ for $\mu$-almost all $\bf x$, as desired.
\end{proof}

This generalizes the fact that for frames over separable Hilbert spaces, having zero redundancy and the canonical dual frame being the unique dual frame are equivalent (which are both in turn equivalent to being a Riesz basis). However, this is not necessarily true for probabilistic frames, since there can be dual probabilistic frames that are not the canonical dual and not of pushforward type. For example, $\delta_1$ is a frame for the real line and the redundancy of $\delta_1$ is zero. Any probability measure $\mu$ with mean value 1 and finite second moment is a dual frame of $\delta_1$ under the coupling $\gamma = (\bI, T)_\#\mu \in \Gamma(\mu, \delta_1)$, where $T:\mathbb{R} \rightarrow \mathbb{R}$ is given by $T(x) = 1, x \in \mathbb{R}$. For example, if we define $\mu = \frac{1}{2}\delta_{\frac{1}{2}} + \frac{1}{2}\delta_{\frac{3}{2}}$, then $\mu$ is a dual frame of $\delta_1$ and not the canonical dual frame. Of course, this does not violate \cref{zeroRedundancy} since there is no map $M: \mathbb{R} \rightarrow \mathbb{R}$ with $\mu = M_\# \delta_1 = \delta_{M(1)}$.

Putting~\cref{thm:finredfinsupp} and~\cref{zeroRedundancy} together with the example above, we see that any probabilistic frame $\mu$ over $\bR^n$ with zero redundancy ($R(\mu) = 0$) can be uniquely associated with a Riesz basis of $\mathbb{R}^n$. Then $\mu$ has a unique dual as a finite frame but only a unique dual of pushforward type as a probabilistic frame.

\section{Approximately Dual Probabilistic Frames}\label{approximateDual}
Approximately dual frames were introduced by Ole Christensen and Richard Laugesen in \cite{christensen2010approximately} and further studied in \cite{javanshiri2016some}. It is usually complicated to calculate a dual frame explicitly and people often seek methods to construct approximate duals that still yield perfect reconstruction. 

This problem is even more acute for probabilistic frames, since characterizing all the dual probabilistic frames for a given probabilistic frame is still a challenge \cite{chen2024general}.  Motivated by this issue, we introduce approximately dual probabilistic frames in this section. In what follows, we interpret $n \times n$ matrices $\bf A$ as operators and take the operator norm of $\bf A$:
$$\|{\bf A} \| =  \underset{{\bf x} \neq 0}{\text{sup}}\ \frac{\|{\bf Ax}\|}{\|{\bf x}\|}.$$

\begin{definition}\label{def:approximate dual}
     Let   $\mu$ be a probabilistic frame. $\nu \in \mathcal{P}_2(\mathbb{R}^n)$ is an \emph{approximately dual probabilistic frame} of $\mu$ if there exists $\gamma \in \Gamma(\mu, \nu)$ such that
     \begin{equation*}
     \Big \| \int_{\mathbb{R}^n \times \mathbb{R}^n} {\bf x}{\bf y}^t d\gamma({\bf x, y}) - {\bf Id} \Big \| < 1 .
\end{equation*}
Furthermore, $\int_{\mathbb{R}^n \times \mathbb{R}^n} {\bf x}{\bf y}^t d\gamma({\bf x, y})$ is called a \emph{mixed frame operator} of $\mu$ and $\nu$.
\end{definition}

Note that for a given probabilistic frame, its dual and approximately dual probabilistic frames and the associated mixed frame operators are dependent on transport couplings. Furthermore, by the symmetry of optimal transport, if $\nu$ is an approximately dual probabilistic frame to $\mu$, then $\mu$ is also an approximately dual probabilistic frame to $\nu$. We call $\mu$ and $\nu$ approximately dual pairs, and we will often shorten \emph{approximately dual probabilistic frame} to \emph{approximately dual frame} or \emph{approximate dual}.  Recall that given a frame, all the associated dual frames, approximately dual frames, and pseudo-dual frames have the same redundancy \cite{bakic2015excesses}. 
However, the redundancies of approximately dual probabilistic frames are not necessarily the same. For example, $\delta_1$ is a probabilistic frame for the real line, and $\nu = \delta_{\frac{1}{2}}$ is an approximate dual of $\delta_1$ with respect to $\delta_1 \otimes \nu$. Now define $\eta  := \frac{1}{2}\delta_{\frac{1}{2}} + \frac{1}{2}\delta_{\frac{1}{3}}$. Then, $\eta$ is also an approximate dual frame of $\delta_1$ with respect to $\delta_1 \otimes \eta$, since
\begin{equation*}
  \left|  \int_{\mathbb{R} \times \mathbb{R}} xz d\delta_1(x) d\eta(z) -1 \right| = \left| \frac{5}{12}-1 \right|= \frac{7}{12}<1.
\end{equation*}
However, $R(\nu) =0 \neq 1 = R(\eta)$. Similarly, the redundancies of dual probabilistic frames are not necessarily the same. Recall that any probability measure $\mu$ with mean value 1 and finite second moment is a dual frame of $\delta_1$ with respect to the coupling $\gamma = (\bI, T)_\#\mu \in \Gamma(\mu, \delta_1)$ where $T:\mathbb{R} \rightarrow \mathbb{R}$ is given by $T(x) = 1$.  Now define $\mu_1: = \frac{1}{2}\delta_{\frac{1}{2}} + \frac{1}{2}\delta_{\frac{3}{2}}$ and $\mu_2 := \frac{1}{3}\delta_{\frac{1}{2}} + \frac{1}{3}\delta_{1} + \frac{1}{3}\delta_{\frac{3}{2}}$. Note that $\mu_1$ and $\mu_2$ are dual probabilistic frames of $\delta_1$, but $R(\mu_1) = 1 \neq 2 = R(\mu_2)$.

In ~\cref{def:approximate dual} we did not require that an approximately dual probabilistic frame actually be a probabilistic frame, though in fact it will be. More than that, there is an uncertainty principle for approximate dual pairs, which was previously studied for probabilistic dual pairs in Corollary 5.4 of \cite{chen2024general}. Furthermore, based on this result, we get the relationship between optimal frame bounds for a probabilistic frame and its approximate duals. The relationship between optimal frame bounds of a frame and its duals in Hilbert spaces has been considered by many authors, for example, in  \cite{razghandi2020characterization}.

\begin{lemma}\label{DualIsFrame}
Let $\mu$ be a probabilistic frame with upper frame bound $B>0$.  Suppose $\nu \in \mathcal{P}_2(\mathbb{R}^n)$ is an approximate dual of $\mu$ with mixed frame operator $\bf A$, then $\nu$ is a probabilistic frame with bounds $ 1/(B \| {\bf A}^{-1}\|^2)$ and $M_2(\nu)$. 

In addition, for any ${\bf f} \in \mathbb{R}^n$, we have the following uncertainty principle:
     \begin{equation*}
        \int_{\mathbb{R}^n} |\langle {\bf A }^{-1} {\bf x}, {\bf f} \rangle |^2 d\mu({\bf x}) \int_{\mathbb{R}^n} |\langle {\bf y}, {\bf f} \rangle |^2 d\nu({\bf y})  \geq \| {\bf f} \|^4,
    \end{equation*}
and equality is attained for all ${\bf f} \in \mathbb{R}^n$  if and only if  ${\bf S}_{\mu}=c{\bf A}{\bf A}^{t}$ and
$\nu=(\frac{1}{c}{\bf A}^{-1})_{\#}\mu$ where $c$ is a nonzero constant.
Furthermore, if $A_\mu$, $B_\mu$, and $A_\nu$, $B_\nu$ are optimal frame bounds for $\mu$ and $\nu$, then 
\begin{equation*}
    A_\nu \geq \frac{1}{B_\mu \| {\bf A}^{-1}\|^2} \ \text{and} \  A_\mu \geq \frac{1}{B_\nu \| {\bf A}^{-1}\|^2}.
\end{equation*}
\end{lemma}

\begin{proof}
 Since $\nu \in \mathcal{P}_2(\mathbb{R}^n)$ is an approximate dual of $\mu$ with mixed frame operator $\bf A$, then  $\bf A$ is invertible with $\norm{\bA - \bI}<1$, and there exists $\gamma \in \Gamma(\mu, \nu)$ such that 
   $$   \int_{\mathbb{R}^n \times \mathbb{R}^n} {\bf x} {\bf y}^t d\gamma({\bf x, y}) = {\bf A}.$$
   Then, for any ${\bf f} \in \mathbb{R}^n$, 
   \begin{equation*}
   \begin{split}
       \|{\bf f}\|^4 = |\langle {\bf f}, {\bf f} \rangle|^2 = |\langle {\bf A}^{-1}{\bf A}{\bf f}, {\bf f} \rangle|^2 &= \left|\int_{\mathbb{R}^n \times \mathbb{R}^n} \langle {\bf A}^{-1} {\bf x}, {\bf f} \rangle \langle {\bf y,f} \rangle d\gamma({\bf x, y}) \right |^2 \\
       & \leq \int_{\mathbb{R}^n} |\langle {\bf A }^{-1} {\bf x}, {\bf f} \rangle |^2 d\mu({\bf x}) \int_{\mathbb{R}^n} |\langle {\bf y}, {\bf f} \rangle |^2 d\nu({\bf y}) \\
       & \leq B \|({\bf A }^{-1})^t {\bf f}\|^2 \int_{\mathbb{R}^n} |\langle {\bf y}, {\bf f} \rangle |^2 d\nu({\bf y}) \\
       &\leq B\|{\bf A }^{-1}\|^2 \| {\bf f}\|^2 \int_{\mathbb{R}^n} |\langle {\bf y}, {\bf f} \rangle |^2 d\nu({\bf y}),
   \end{split}
   \end{equation*}
where the first inequality is due to the Cauchy–Schwarz inequality and the second follows from the frame property of $\mu$. Note that we have proved the desired uncertainty principle along the way. Therefore, for any ${\bf f} \in \mathbb{R}^n$,
\begin{equation*}
       \frac{\|{\bf f}\|^2}{B \|{\bf A }^{-1}\|^2}
       \leq  \int_{\mathbb{R}^n} |\langle {\bf y}, {\bf f} \rangle |^2 d\nu({\bf y}) \leq M_2(\nu)\|{\bf f}\|^2,
   \end{equation*}
 which implies that $\nu$ is a probabilistic frame with bounds $ 1/(B \| {\bf A}^{-1}\|^2)$ and $M_2(\nu)$. 

In addition, for a fixed ${\bf f}$, equality holds in the uncertainty principle if and only if the two scalar functions
\[
({\bf x},{\bf y})\longmapsto \langle {\bf A}^{-1}{\bf x},{\bf f}\rangle
\quad\text{and}\quad
({\bf x},{\bf y})\longmapsto \langle {\bf y},{\bf f}\rangle
\]
are linearly dependent in $L^{2}(\gamma, \mathbb{R}^n \times \mathbb{R}^n)$; that is, there exists a scalar
$c({\bf f})$ that may depend on ${\bf f}$ such that
\begin{equation}\label{eq:inner product ae}
\langle {\bf A}^{-1}{\bf x},{\bf f}\rangle
= c({\bf f}) \langle {\bf y},{\bf f}\rangle
\ \text{for $\gamma$-almost every }({\bf x},{\bf y}).
\end{equation}

We claim that if the equality is true for all $\mathbf{f} \in \mathbb{R}^n$, then $c({\bf f})$ is a constant, independent of ${\bf f}$. 

To see this, note first that $c(\alpha {\bf f}) = c({\bf f})$ for all $\alpha \neq 0$. 

Now, let $E$ be the set on which~\eqref{eq:inner product ae} holds. Then $E$ has full $\gamma$-measure and hence the projection $\pi_y(E)$ has full $\nu$-measure. Therefore, by \cref{cor:full measure contains basis} there exist ${\bf x}_1, \dots , {\bf x}_n, {\bf y}_1, \dots {\bf y}_n \in \R^n$ so that ${\bf y}_1, \dots {\bf y}_n$ is a basis for $\R^n$ and $({\bf x}_i, {\bf y}_i) \in E$ for all $i=1, \dots , n$.

For any $i=1, \dots , n$, suppose ${\bf f} \in \R^n$ is a unit vector with $\langle {\bf y}_i, {\bf f} \rangle \neq 0$. Write 
\[
    {\bf f} = \alpha {\bf y}_i + {\bf g},
\]
where $\alpha \neq 0$ and ${\bf g} \in {\bf y}_i^\perp$. Then
\[
\langle {\bf A}^{-1}{\bf x}_i,\alpha {\bf y}_i\rangle
= c(\alpha {\bf y}_i) \langle {\bf y}_i,\alpha {\bf y}_i\rangle = c({\bf y}_i) \langle {\bf y}_i,\alpha {\bf y}_i\rangle
\]
and
\[
    \langle {\bf A}^{-1}{\bf x}_i,{\bf g}\rangle = c({\bf g}) \langle {\bf y}_i,{\bf g}\rangle = 0
\]
Therefore,
\[
    c({\bf f}) \alpha \|{\bf y}_i\|^2 = c({\bf f})\langle {\bf y}_i, {\bf f}\rangle = \langle {\bf A}^{-1} {\bf x}_i , {\bf f} \rangle = \langle {\bf y}_i , c({\bf y}_i) \alpha {\bf y}_i\rangle = c({\bf y}_i) \alpha \|{\bf y}_i\|^2,
\]
so $c({\bf f}) = c({\bf y}_i)$. 

In other words, $c$ is constant on unit vectors in $\R^n \backslash{\bf y}_i^\perp$ and hence, by the scale-invariance of $c$, on all of each $\R^n \backslash{\bf y}_i^\perp$. Since these sets have nontrivial intersections, we see that all constants are the same, so $c$ is constant on 
$$\bigcup_{i=1}^n \left(\R^n \backslash {\bf y}_i^\perp\right) = \R^n \backslash \left(\bigcap_{i=1}^n {\bf y}_i^\perp\right).
$$
Since ${\bf y}_1, \dots {\bf y}_n$ is a basis for $\R^n$, the intersection $\bigcap_{i=1}^n {\bf y}_i^\perp = \{{\bf 0}\}$, so $c$ is constant on $\R^n \backslash \{{\bf 0}\}$.

Moreover, re-defining $c({\bf 0})$ to equal this common constant doesn't affect the validity of~\eqref{eq:inner product ae}, so we see that $c$ can be chosen to be constant on all of $\R^n$.

Therefore, we have shown if equality in the uncertainty principle is true for all ${\bf f} \in \mathbb{R}^n$,  then
$$\langle {\bf A}^{-1}{\bf x},{\bf f}\rangle
= c \langle {\bf y},{\bf f}\rangle
\ \text{for $\gamma$-almost every }({\bf x},{\bf y}),$$ which implies that $\gamma = (\mathbf{Id}, \frac{1}{c}{\bf A}^{-1})_\#\mu$. Since $\gamma \in \Gamma(\mu, \nu)$, then $\nu = (\frac{1}{c}{\bf A}^{-1})_\#\mu$. Also $\int_{\mathbb{R}^n \times \mathbb{R}^n} {\bf x} {\bf y}^t d\gamma({\bf x, y}) = {\bf A}$ implies that ${\bf S}_{\mu}=c{\bf A}{\bf A}^{t}$.

On the other hand, suppose ${\bf S}_{\mu}=c{\bf A}{\bf A}^{t}$ and
$\nu=(\frac{1}{c}{\bf A}^{-1})_{\#}\mu$ for some nonzero constant $c$.
Then $\nu$ is an approximately dual frame of $\mu$ with mixed frame operator ${\bf A}$ against $\gamma = (\mathbf{Id}, \frac{1}{c}{\bf A}^{-1})_\#\mu \in \Gamma(\mu, \nu)$, which follows from 
   $$   \int_{\mathbb{R}^n \times \mathbb{R}^n} {\bf x} {\bf y}^t d\gamma({\bf x, y}) = \frac{1}{c} \int_{\mathbb{R}^n} {\bf x} {\bf x}^t d\mu({\bf x}) ({\bf A}^{-1})^t = \frac{1}{c} {\bf S}_{\mu} ({\bf A}^{-1})^t = {\bf A}.
   $$
Since  $\gamma = (\mathbf{Id}, \frac{1}{c}{\bf A}^{-1})_\#\mu$, then for $\gamma$-almost every $({\bf x},{\bf y})$, ${\bf y}=\frac{1}{c}{\bf A}^{-1}{\bf x}$. Thus for any fixed ${\bf f}$, 
$\langle {\bf A}^{-1}{\bf x},{\bf f}\rangle
= c \langle {\bf y},{\bf f}\rangle$, for $\gamma$-almost every $({\bf x},{\bf y})$. 
Therefore, equality in the uncertainty principle holds for all $\mathbf{f} \in \mathbb{R}^n$.

 Finally, if $A_\mu$, $B_\mu$ and $A_\nu$, $B_\nu$ are optimal bounds for $\mu$ and $\nu$, then the above result gives $ A_\nu \geq \frac{1}{B_\mu \| {\bf A}^{-1}\|^2}$, and similarly,  for any ${\bf f} \in \mathbb{R}^n$, we have 
   \begin{equation*}
   \begin{split}
       \|{\bf f}\|^4  = |\langle ({\bf A}^{-1})^t{\bf A}^t{\bf f}, {\bf f} \rangle|^2 &= \left|\int_{\mathbb{R}^n \times \mathbb{R}^n} \langle ({\bf A}^{-1})^t {\bf y}, {\bf f} \rangle \langle {\bf x,f} \rangle d\gamma({\bf x, y}) \right |^2 \\
       & \leq  \int_{\mathbb{R}^n} |\langle {\bf y}, {\bf A}^{-1} {\bf f} \rangle |^2 d\nu({\bf y}) \int_{\mathbb{R}^n} |\langle {\bf x}, {\bf f} \rangle |^2 d\mu({\bf x})\\
       &\leq B_\nu \|{\bf A }^{-1}\|^2 \| {\bf f}\|^2 \int_{\mathbb{R}^n} |\langle {\bf x}, {\bf f} \rangle |^2 d\mu({\bf x}).
   \end{split}
   \end{equation*}
Therefore, $\frac{1}{B_\nu \| {\bf A}^{-1}\|^2}$ is a lower frame bound for $\mu$ and thus $A_\mu \geq \frac{1}{B_\nu \| {\bf A}^{-1}\|^2}$.
\end{proof}

If $\nu \in \mathcal{P}_2(\mathbb{R}^n)$ is a dual probabilistic frame of $\mu$, then the mixed frame operator ${\bf A} = {\bf Id}$, and thus for any ${\bf f}$ in the unit sphere $ \mathbb{S}^{n-1}$, the uncertainty principle is 
     \begin{equation*}
        \int_{\mathbb{R}^n} |\langle {\bf x}, {\bf f} \rangle |^2 d\mu({\bf x}) \int_{\mathbb{R}^n} |\langle {\bf y}, {\bf f} \rangle |^2 d\nu({\bf y})  \geq 1.
    \end{equation*}
Furthermore, if $A_\mu$, $B_\mu$, and $A_\nu$, $B_\nu$ are optimal frame bounds for $\mu$ and $\nu$, then 
\begin{equation*}
    A_\nu \geq \frac{1}{B_\mu } \ \text{and} \ A_\mu \geq \frac{1}{B_\nu},
\end{equation*}
where equality holds if $\nu$ is the canonical dual frame of $\mu$. 

\begin{lemma}
    Let $\mu$ be a probabilistic frame.  The set of approximately dual probabilistic frames of $\mu$ is convex.
\end{lemma}
\begin{proof}
Suppose $\nu_1$ and $ \nu_2$ are approximately dual probabilistic frames of $\mu$ with respect to $\gamma_1 \in \Gamma(\mu, \nu_1)$ and  $\gamma_2 \in \Gamma(\mu, \nu_2)$. Let $0 \leq w \leq 1$ and define $\nu := w \nu_1 + (1-w)\nu_2 $ and $ \gamma := w\gamma_1+(1-w)\gamma_2 \in \Gamma(\mu, \nu)$. Then
    \begin{equation*}
    \begin{split}
        \Big \| \int_{\mathbb{R}^n \times \mathbb{R}^n} {\bf x}{\bf y}^t d\gamma({\bf x, y}) - {\bf Id} \Big \| &\leq   w \Big \| \int_{\mathbb{R}^n \times \mathbb{R}^n} {\bf x}{\bf y}^t d\gamma_1({\bf x, y}) - {\bf Id} \Big \| \\
        & +(1-w) \Big \| \int_{\mathbb{R}^n \times \mathbb{R}^n} {\bf x}{\bf y}^t d\gamma_2({\bf x, y}) - {\bf Id} \Big \| < 1. 
    \end{split}
\end{equation*}
Therefore, $\nu = w \nu_1 + (1-w)\nu_2$ is an approximately dual frame of $\mu$. 
\end{proof}

One might hope that the convexity of the set of approximately dual frames for a given probabilistic frame would allow us to simply find the extreme points and use the Krein–Milman theorem to characterize the set.  However, Krein–Milman requires the convex set to additionally be compact, and the set of approximately dual probabilistic frames for a given probabilistic frame is not even weakly closed in $\mathcal{P}_2(\mathbb{R}^n)$. To show this, we return to our favorite example of $\delta_1$ on the real line. For each $n>0$, define $\nu_n = \delta_{\frac{1}{n+1}}$. Then, for any $n>0$, $\nu_n$ is an approximately dual frame to $\delta_1$ with respect to the product measure $\delta_1 \otimes \nu_n$, since
\begin{equation*}
    \left | \int x y d\delta_1(x)d\nu(y)  -1 \right | = 1- \frac{1}{n+1} <1.
\end{equation*}
Note that $\nu_n$ converges weakly to $\delta_0$. However,  $\delta_0$ is not a probabilistic frame (since $\int_{\bR} y y d \delta_0(y) =0$) and thus not an approximately dual frame of $\delta_1$. Then, the set of approximately dual frames for $\delta_1$ is not weakly closed and thus not weakly compact. 

Although it is difficult to characterize all the approximate duals, the following results show that given an approximately dual frame,  we can get one dual frame and construct infinitely many approximately dual frames with arbitrary closeness to the perfect reconstruction. 

\begin{lemma}
    Let $\mu$ be a probabilistic frame.  If $\nu \in \mathcal{P}_2(\mathbb{R}^n)$ is an approximate dual frame of $\mu$ and ${\bf A}$ is the transpose of their mixed frame operator, then ${{\bf A}^{-1}}_\#\nu$ is a dual probabilistic frame to $\mu$.
\end{lemma}
The above lemma follows from a more general result concerning pseudo duals in ~\cref{lem:pushpseudo}.

The following proposition is about the construction of approximately dual frames with arbitrary reconstruction resolution.

\begin{proposition}
    Let $\mu$ be a probabilistic frame and let $\nu \in \mathcal{P}_2(\mathbb{R}^n)$ be an approximate dual frame of $\mu$ with respect to the coupling  $\gamma \in \Gamma(\mu, \nu)$. Suppose ${\bf A}$ is the transpose of the associated mixed frame operator. Then 
     \begin{enumerate}
         \item \label{it:dual exp} The dual probabilistic frame ${{\bf A}^{-1}}_\#\nu$ of $\mu$ can be written as 
         \begin{equation*}
             {{\bf A}^{-1}}_\#\nu = ({\bf Id} + \sum_{k=1}^{\infty} ({\bf Id} - {\bf A})^k)_\# \nu.
         \end{equation*}
         \item \label{it:arbitrary resolution} For any fixed $N \in \mathbb{N}$, consider the partial sum
         \begin{equation*}
             \nu_N := \big (\sum_{k=0}^{N} ({\bf Id} - {\bf A})^k \big)_\# \nu  = \big({\bf Id} + \sum_{k=1}^{N} ({\bf Id} - {\bf A})^k \big)_\# \nu.
         \end{equation*}
         Then $ \nu_N$ is an approximately dual probabilistic frame to $\mu$ with respect to $\Tilde{\gamma}_N:=\big ({\bf Id}, {\bf Id} + \sum\limits_{k=1}^{N} ({\bf Id} - {\bf A})^k \big)_\#\gamma \in \Gamma(\mu, \nu_N)$, and 
         \begin{equation*}
          \Big \| \int_{\mathbb{R}^n \times \mathbb{R}^n} {\bf y} {\bf x}^t d\Tilde{\gamma}_N({\bf x, y}) - {\bf Id} \Big \|  \leq \| {\bf Id} -{\bf A}  \|^{N+1} \rightarrow 0 \ \text{as $N \rightarrow \infty$}.
\end{equation*}
     \end{enumerate}
\end{proposition}
\begin{proof}
The inverse of ${\bf A}$ can be written as a Neumann series
\begin{equation*}
    {{\bf A}^{-1}} = ({\bf Id} - ({\bf Id} - {\bf A}))^{-1} = \sum_{k=0}^{\infty} ({\bf Id} - {\bf A})^k = {\bf Id} + \sum_{k=1}^{\infty} ({\bf Id} - {\bf A})^k.
\end{equation*}
Hence, (\ref{it:dual exp}) follows. For (\ref{it:arbitrary resolution}), note that $\nu_N \in \mathcal{P}_2(\mathbb{R}^n)$ and 
\begin{equation*}
\begin{split}
    \int_{\mathbb{R}^n \times \mathbb{R}^n} {\bf y} {\bf x}^t d\Tilde{\gamma}_N({\bf x, y}) 
    &= \int_{\mathbb{R}^n \times \mathbb{R}^n} \left(\sum\limits_{k=0}^{N} ({\bf Id} - {\bf A})^k \right){\bf y} {\bf x}^t d\gamma({\bf x, y}) \\
    &= \sum\limits_{k=0}^{N} ({\bf Id} - {\bf A})^k \int_{\mathbb{R}^n \times \mathbb{R}^n} {\bf y} {\bf x}^t d\gamma({\bf x, y})= \sum\limits_{k=0}^{N} ({\bf Id} - {\bf A})^k {\bf A} \\
    &= \sum\limits_{k=0}^{N} ({\bf Id} - {\bf A})^k ({\bf Id} -({\bf Id} - {\bf A})) = {\bf Id} - ({\bf Id} -{\bf A})^{N+1}.
\end{split}
\end{equation*}
Since $\| {\bf Id} -{\bf A}  \| < 1$, we have 
 \begin{equation*}
 \begin{split}
      \norm{\int_{\mathbb{R}^n \times \mathbb{R}^n} {\bf y} {\bf x}^t d\Tilde{\gamma}_N({\bf x, y}) - {\bf Id}} &=   \norm{({\bf Id} -{\bf A})^{N+1}}  \leq \norm{{\bf Id} -{\bf A}  }^{N+1} < 1. 
 \end{split}
\end{equation*}
Thus, $ \nu_N$ is an approximately dual probabilistic frame to $\mu$ with respect to $\Tilde{\gamma}_N$, and 
 \begin{equation*}
          \Big \| \int_{\mathbb{R}^n \times \mathbb{R}^n} {\bf y} {\bf x}^t d\Tilde{\gamma}_N({\bf x, y}) - {\bf Id} \Big \|  \leq \| {\bf Id} -{\bf A}  \|^{N+1} \rightarrow 0 \ \text{as $N \rightarrow \infty$}.
\end{equation*}
\end{proof}

Since we can construct approximately dual frames with arbitrarily small error, we give another alternative definition of the approximately dual probabilistic frame. Let $\mu$ be a probabilistic frame and $0<\epsilon<1$. A probability measure $\nu \in \mathcal{P}_2(\mathbb{R}^n)$ is an \emph{$\epsilon$-approximately dual probabilistic frame} of $\mu$ if there exists $\gamma \in \Gamma(\mu, \nu)$ such that
     \begin{equation*}
     \norm{\int_{\mathbb{R}^n \times \mathbb{R}^n} {\bf x}{\bf y}^t d\gamma({\bf x, y}) - {\bf Id}}\leq \epsilon .
\end{equation*}

Similar to the situation with dual probabilistic frames (see \cref{dualChar}), we can characterize all approximately dual frames of pushforward type. We begin with the following lemma.
\begin{lemma}
    Suppose $\bf A$ is any $n \times n$ matrix such that $\|{\bf A -Id} \| <1$. Then $({\bf A}^t{\bf S}_\mu^{-1})_\#\mu$ is an approximately dual frame of $\mu$. 
\end{lemma}
\begin{proof}
    Define $\gamma:= ({\bf Id}, {\bf A}^t{\bf S}_\mu^{-1})_\#\mu \in \Gamma(\mu, ({\bf A}^t{\bf S}_\mu^{-1})_\#\mu)$, then
    \begin{equation*}
        \int_{\mathbb{R}^n \times \mathbb{R}^n} {\bf x} {\bf y}^t d\gamma({\bf x, y}) = \int_{\mathbb{R}^n } {\bf x} {\bf x}^t {\bf S}_\mu^{-1} {\bf A} d\mu({\bf x}) ={\bf S}_\mu {\bf S}_\mu^{-1} {\bf A}  ={\bf A}.
    \end{equation*}
    Since $\|{\bf A -Id} \| <1$, then $({\bf A}^t{\bf S}_\mu^{-1})_\#\mu$ is an approximately dual frame to $\mu$. 
\end{proof}

 We characterize all approximately dual frames of pushforward type as follows.
 
\begin{corollary}\label{PushforwardCharacter}
    Let $T_\# \mu$ be an approximately dual frame to $\mu$ where $T: \mathbb{R}^n \rightarrow \mathbb{R}^n$ is a measurable map. Then for any $\ {\bf x} \in \mathbb{R}^n$, $T: \mathbb{R}^n \rightarrow \mathbb{R}^n$ precisely satisfies 
    \begin{equation*}
        T({\bf x}) = {\bf A}^t{\bf S}_\mu^{-1}{\bf x} + h({\bf x}),
    \end{equation*}
    where $\bf A$ is an $n \times n$ matrix satisfying $\|{\bf A -Id} \| <1$ and $h: \mathbb{R}^n \rightarrow \mathbb{R}^n$ is such that $h_\#\mu \in \mathcal{P}_2(\mathbb{R}^n)$ with $ \int_{\mathbb{R}^n } {\bf x} h({\bf x})^t d\mu({\bf x}) = {\bf 0}_{n \times n}$.
\end{corollary}
This may be viewed as a corollary of a result concerning pseudo-dual probabilistic frames.  See~\cref{cor:psuedodualpush} for the proof.

Using a similar argument as in the proof, we give another characterization.
\begin{corollary}
    Let $T_\# \mu$ be an approximately dual frame to $\mu$ where $T: \mathbb{R}^n \rightarrow \mathbb{R}^n$ is measurable. Then for all $\ {\bf x} \in \mathbb{R}^n$ the map $T: \mathbb{R}^n \rightarrow \mathbb{R}^n$ precisely satisfies  
    \begin{equation*}
        T({\bf x}) = {\bf A}^t{\bf S}_\mu^{-1}({\bf x}) + h({\bf x}) - \int_{\mathbb{R}^n} \langle {\bf S}_\mu^{-1}{\bf x}, {\bf y} \rangle h({\bf y}) d\mu({\bf y}),
      \end{equation*}
where $\bf A$ is an $n \times n$ matrix satisfying $\|{\bf A -Id} \| <1$ and $h: \mathbb{R}^n \rightarrow \mathbb{R}^n$ is such that $h_\#\mu \in \mathcal{P}_2(\mathbb{R}^n)$.
\end{corollary}

Note that not every approximately dual frame is of pushforward type, since one can find a dual frame that is not of the pushforward type.  As we've already seen, $\mu = \frac{1}{2}\delta_{\frac{1}{2}} + \frac{1}{2}\delta_{\frac{3}{2}}$ is a dual frame of $\delta_1$ under the coupling $\gamma = (\bI, T)_\#\mu \in \Gamma(\mu, \delta_1)$ where $T(x) = 1, x \in \mathbb{R}$ even though there does not exist a map $M: \mathbb{R} \rightarrow \mathbb{R}$ such that $\mu = M_\# \delta_1 = \delta_{M(1)}$. We finish this section with the following corollary. 

\begin{corollary}
        Let $\mu$ be a probabilistic frame and $\bf A$ an $n \times n$ matrix such that $\|{\bf A -Id} \| <1$.  Suppose $T_\# \mu$ is a dual probabilistic frame to $\mu$ where $T: \mathbb{R}^n \rightarrow \mathbb{R}^n$ is measurable. Then $({\bf A}^t{\bf S}_\mu^{-1}-{\bf Id}+{\bf S}_\mu T)_\# \mu$ is an approximately dual frame of $\mu$.
\end{corollary}
\begin{proof}
Since $T_\# \mu$ is a dual probabilistic frame to $\mu$, then $\int_{\mathbb{R}^n } {\bf x} T({\bf x})^t d\mu({\bf x}) = {\bf Id}$.
By letting $h({\bf x}) = (-{\bf Id}+{\bf S}_\mu T){\bf x}$, then $h_\#\mu \in \mathcal{P}_2(\mathbb{R}^n)$ and 
     \begin{equation*}
        \int_{\mathbb{R}^n } {\bf x} h({\bf x})^t d \mu({\bf x}) = - \int_{\mathbb{R}^n } {\bf x} {\bf x}^t d\mu({\bf x}) + \int_{\mathbb{R}^n } {\bf x} T({\bf x})^t d\mu({\bf x})\cdot {\bf S}_\mu = {\bf 0}_{n \times n}.
    \end{equation*}
   By \cref{PushforwardCharacter}, $({\bf A}^t{\bf S}_\mu^{-1}-{\bf Id}+{\bf S}_\mu T)_\# \mu$ is an approximate dual frame of $\mu$. 
\end{proof}

\section{Approximate Duals of Perturbed Probabilistic Frames }\label{peturbation}
In this section, we consider the approximately dual frames of perturbed probabilistic frames. We claim that if a probability measure is close to one of the probabilistic frames in some dual pair, then the other probabilistic frame in the dual pair is an approximate dual to this probability measure, which is explained in \cref{pertubedFrame}, \cref{corollary:DualFrameOpenness}, and the following diagram.
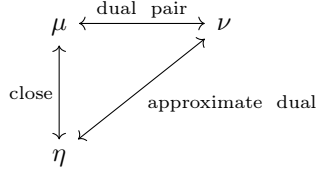
\begin{center}
  \begin{tikzcd}[row sep =40 pt, column sep=huge]
  \mu \arrow[leftrightarrow, r, "\text{dual \ pair}"]  & \nu \arrow[leftrightarrow, ld, "\text{approximate \ dual}"] \\
 \eta \arrow[u,leftrightarrow, "\text{close}" ] 
 \end{tikzcd}
\end{center}

Note that the space of probabilistic frames is open in the $2$-Wasserstein topology, meaning that small perturbations of probabilistic frames are still probabilistic frames. More precisely, we have the following proposition.

\begin{proposition}[{\cite[Proposition 1.2, Corollary 1.3]{chen2023paley}}]\label{prop:quadratic closeness}
    Let $\mu$ be a probabilistic frame with lower frame bound $A > 0$. If there exist $\eta \in \mathcal{P}_2(\mathbb{R}^n)$  and $\gamma \in \Gamma(\eta, \mu) $ so that 
    \[
       \lambda:=  \int_{\mathbb{R}^n \times \mathbb{R}^n} \| {\bf x}- {\bf y} \|^2 d\gamma({\bf x},{\bf y})  < A,
    \]
then $\eta$ is  a probabilistic frame with bounds $(\sqrt{A}-\sqrt{\lambda})^2 $ and $M_2(\eta)$. Furthermore, if $W_2(\mu, \eta) < \sqrt{A}$,
then $\eta$ is a probabilistic frame with bounds $(\sqrt{A}-W_2(\mu, \nu))^2 $ and $M_2(\eta)$.
\end{proposition}

Our first result in this section is motivated by the study of topology on the set of approximately dual frames, which is generally not weakly closed in $\mathcal{P}_2(\mathbb{R}^n)$. However, we claim that given an approximate dual pair $\mu$ and $\nu$, if a probability measure $\eta$ is close to $\mu$ in an appropriate sense, then $\eta$ is an approximate dual frame of $\nu$, as depicted in the following diagram.   
\begin{center}
  \begin{tikzcd}[column sep=80pt,row sep=40 pt]
  \mu \arrow[leftrightarrow, r, "\text{approximate \ dual}"]  & \nu \arrow[leftrightarrow, ld, "\text{approximate \ dual}"] \\
 \eta \arrow[u, " \text{``close"} " ] 
 \end{tikzcd}
\end{center}

\begin{proposition}\label{proposition:ApproDualPertubation}
    Let $\mu$ be a probabilistic frame with lower bound $A>0$ and $\nu$ an approximately dual frame of $\mu$ with mixed frame operator $ {\bf A}$.  If the upper frame bound of $\nu$  is $C>0$ 
 and there exist $\eta \in \mathcal{P}_2(\mathbb{R}^n)$ and $\gamma \in \Gamma(\eta, \mu)$ such that 
$$\int_{\mathbb{R}^n \times \mathbb{R}^n} \| {\bf x}-{\bf A^{-1} y} \|^2 d\gamma({\bf x, y}) <  \frac{1}{C},$$
then $\eta $ is a probabilistic frame and $\nu$ is an approximately dual frame to $\eta$.
\end{proposition}
\begin{proof}
    Since $\nu$ is an approximately dual frame to $\mu$ with mixed frame operator $ {\bf A}$, then there exists $\pi \in \Gamma(\mu, \nu)$ such that 
    \begin{equation*}
        {\bf A} = \int_{\mathbb{R}^n \times \mathbb{R}^n} {\bf y}{\bf z}^t d\pi({\bf y, z})
    \end{equation*}
and $\|{\bf A - Id}  \|<1$. Therefore, ${\bf A }$ is invertible. By \cref{gluinglemma} (Gluing Lemma), there exists $\Tilde{\pi} \in \mathcal{P}(\mathbb{R}^n \times \mathbb{R}^n \times \mathbb{R}^n)$ such that $\gamma = {\pi_{xy}}_{\#}{\Tilde{\pi}} $ and $\pi = {\pi_{yz}}_{\#}{\Tilde{\pi}} $. Now define $\Tilde{\gamma} = {\pi_{xz}}_{\#}{\Tilde{\pi}} \in \Gamma(\eta, \nu)$.
Then, for any ${\bf f} \in \mathbb{R}^n$, 
 \begin{equation*}
 \begin{split}
     \norm{\int_{\mathbb{R}^n \times \mathbb{R}^n} {\bf x} \langle {\bf z, f} \rangle d\Tilde{\gamma}({\bf x, z}) - {\bf f}}^2
     &= \norm{\int_{\mathbb{R}^n \times \mathbb{R}^n} {\bf x} \langle {\bf z, f} \rangle d\Tilde{\gamma}({\bf x, z}) - {\bf A^{-1} A f} }^2 \\
      &=  \norm{ \int_{\mathbb{R}^n \times \mathbb{R}^n} {\bf x} \langle {\bf z, f} \rangle  d\Tilde{\gamma}({\bf x, z}) - \int_{\mathbb{R}^n \times \mathbb{R}^n} {\bf A^{-1}} {\bf y} \langle {\bf z, f} \rangle   d\pi({\bf y, z})}^2 \\
      &= \norm{ \int_{\mathbb{R}^n \times \mathbb{R}^n \times \mathbb{R}^n} ({\bf x} - {\bf A^{-1}} {\bf y}) \langle {\bf z, f} \rangle   d\Tilde{\pi}({\bf x, y, z}) }^2 \\
      & \leq \left | \int_{\mathbb{R}^n \times \mathbb{R}^n \times \mathbb{R}^n}  \| {\bf x}-{\bf A^{-1} y} \| \,  | \langle {\bf z, f} \rangle | \,d\Tilde{\pi}({\bf x, y, z}) \right |^2 \\
      & \leq  \int_{\mathbb{R}^n \times \mathbb{R}^n} \| {\bf x}-{\bf A^{-1} y} \|^2  d\gamma({\bf x, y}) \int_{\mathbb{R}^n}  | \langle {\bf z, f} \rangle |^2   d\nu({\bf z})   \\
      & \leq  \left(\int_{\mathbb{R}^n \times \mathbb{R}^n} \| {\bf x}-{\bf A^{-1} y} \|^2  d\gamma({\bf x, y})\right) C \|{\bf  f} \|^2  < \Vert  {\bf f} \Vert^2,
 \end{split}
    \end{equation*}
where the second inequality is due to the Cauchy–Schwarz inequality. Therefore,
\begin{equation*}
    \norm{\int_{\mathbb{R}^n \times \mathbb{R}^n} {\bf x} {\bf z}^t d\Tilde{\gamma}({\bf x, z}) - {\bf Id} } < 1.
\end{equation*}
Thus, $\nu $ is an approximately dual frame to $\eta$ with respect to $\Tilde{\gamma} \in \Gamma(\eta, \nu)$. \cref{DualIsFrame} then implies that $\eta$ is also a probabilistic frame. 
\end{proof}

Letting $\nu$ be a dual frame of $\mu$ (i.e., $\bf A = Id$) immediately implies the following.
\begin{corollary}\label{pertubedFrame}
Let $\mu$ be a probabilistic frame with lower frame bound $A>0$. Suppose $\nu$ is a dual probabilistic frame to $\mu$ with upper frame bound $C>0$ with $AC\leq 1$.  If there exist $\eta \in \mathcal{P}_2(\mathbb{R}^n)$  and $\gamma \in \Gamma(\eta, \mu) $ such that
    \begin{equation*}
        \int_{\mathbb{R}^n \times \mathbb{R}^n} \| {\bf x}- {\bf y} \|^2 d\gamma({\bf x},{\bf y})  < A,
    \end{equation*}
 then $\nu$ is an approximately dual probabilistic frame to $\eta$.
\end{corollary}

If $\gamma \in \Gamma(\eta, \mu)$ in the above corollary is an optimal transport coupling for $W_2(\eta, \mu)$, then we have the following corollary.
\begin{corollary}\label{corollary:DualFrameOpenness}
    Let $\mu$ be a probabilistic frame with lower bound $A>0$ and  $\nu$ a dual frame to $\mu$ with upper bound $C>0$ so that $AC\leq 1$.  If there exist $\eta \in \mathcal{P}_2(\mathbb{R}^n)$ such that $W_2(\eta, \mu)  < \sqrt{A},$
 then $\nu$ is an approximately dual frame to $\eta$. Consequently, for a fixed probabilistic frame, its dual probabilistic frames are interior points in the set of approximately dual frames in the $2$-Wasserstein topology.
\end{corollary}

One might be concerned by the condition that the lower frame bound $A$ and the upper frame bound $C$ for the dual frame need to satisfy $AC \leq 1$. After all, maybe there are no such dual frames or, even if there are, maybe they are effectively impossible to find. Fortunately, the canonical dual frame satisfies this condition.   

\begin{corollary}
    Let $\mu$ be a probabilistic frame with lower frame bound $A>0$. Suppose there exist $\eta \in \mathcal{P}_2(\mathbb{R}^n)$  and $\gamma \in \Gamma(\eta, \mu) $ such that
    \begin{equation*}
        \int_{\mathbb{R}^n \times \mathbb{R}^n} \| {\bf x}- {\bf y} \|^2 d\gamma({\bf x},{\bf y})  < A.
    \end{equation*}
 Then ${{\bf S}_{\mu}^{-1}}_{\#}\mu$ is an approximately dual frame to $\eta$.
\end{corollary}

\begin{proof}
${{\bf S}_{\mu}^{-1}}_{\#}\mu$ is a dual probabilistic frame of $\mu$ with upper frame bound $\frac{1}{A}$. Letting $\nu = {{\bf S}_{\mu}^{-1}}_{\#}\mu$, we can apply \cref{pertubedFrame} since $C=\frac{1}{A}$ and $AC=1$, so the result follows. 
\end{proof}

In fact, the condition in \cref{pertubedFrame} that $A$ is the lower frame bound of $\mu$ is not necessary: we used it to guarantee that $\eta$ and $\mu$ are close enough that we can invoke \cref{prop:quadratic closeness} and conclude $\eta$ is a probabilistic frame. But we can use \Cref{DualIsFrame} to see that $\eta$ is a probabilistic frame instead. 

More precisely, if we have 
    \begin{equation*}
        \int_{\mathbb{R}^n \times \mathbb{R}^n} \| {\bf x}- {\bf y} \|^2 d\gamma({\bf x},{\bf y})  < A  \enskip\text{or} \enskip  W_2(\eta, \mu)  < \sqrt{A}
    \end{equation*}
and $AC \leq 1$, then the proof of \cref{proposition:ApproDualPertubation} for $\bf A =Id$ shows that
\begin{equation*}
    \Big \| \int_{\mathbb{R}^n \times \mathbb{R}^n} {\bf x} {\bf z}^t d\Tilde{\gamma}({\bf x, z}) - {\bf Id} \Big \| < \sqrt{AC} \leq 1, 
\end{equation*}
which implies that $\nu$ and $\eta$ form an approximate dual pair. And then \cref{DualIsFrame} tells us that $\eta$ is also a probabilistic frame.

Therefore, we can construct discrete approximately dual frames for arbitrary probabilistic frames as follows. Suppose $\eta$ is a probabilistic frame and let $\mu = \frac{1}{N}\sum_{k=1}^N \delta_{{\bf f}_k}$ be a discrete approximation of $\eta$. Since Dirac delta measures are dense in $\mathcal{P}_2(\mathbb{R}^n)$, we can find such a $\mu =\frac{1}{N}\sum\limits_{k=1}^N \delta_{{\bf f}_k}$ so that $W_2(\eta, \mu) < \sqrt{A}_N$ for any error bound $\sqrt{A}_N > 0$. 
Now, if $\nu= \frac{1}{N}\sum\limits_{k=1}^N \delta_{{\bf g}_k}$ is a dual frame to $\mu$ with upper frame bound $C_N$ so that $A_NC_N \leq 1$, then the above reasoning shows that $\nu$ is an approximately dual frame to the original probabilistic frame $\eta$. We record this observation in the following diagram and theorem.

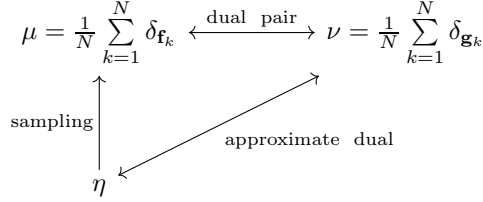
\begin{center}
  \begin{tikzcd}[row sep =huge, column sep=huge]
  \mu = \frac{1}{N}\sum\limits_{k=1}^N \delta_{{\bf f}_k} \arrow[leftrightarrow, r, "\text{dual \ pair}"]  & \nu= \frac{1}{N}\sum\limits_{k=1}^N \delta_{{\bf g}_k} \arrow[leftrightarrow, ld, "\text{approximate \ dual}"] \\
 \eta \arrow[u, "\text{sampling}" ] 
 \end{tikzcd}
\end{center}

\begin{theorem}\label{thm:discrete approx dual}
Given $0<A_N< \infty$, let $\eta \in \mathcal{P}_2(\mathbb{R}^n)$ and let $\mu =\frac{1}{N} \sum_{k=1}^N \delta_{{\bf f}_k}$ be a probabilistic frame such that  
    \begin{equation*}
        W_2(\eta, \mu)  < \sqrt{A_N}.
    \end{equation*}
    Suppose $\nu= \frac{1}{N}\sum_{k=1}^N \delta_{{\bf g}_k}$ is a dual probabilistic frame to $\mu$ with upper bound $C_N$ such that $A_N C_N\leq 1$.  
 Then, $\eta$ is a probabilistic frame and $\nu=\frac{1}{N}\sum_{k=1}^N \delta_{{\bf g}_k}$ is an approximately dual frame of $\eta$.
\end{theorem}

Indeed, such $\nu$ always exists. To see this, note that we can pick $\mu$ and $A_N$ such that $W_2(\eta, \mu) < \sqrt{A_N} \leq \frac{\sqrt{A_\eta}}{2}$ where $A_\eta>0$ is the lower frame bound of $\eta$, and hence \cref{prop:quadratic closeness} implies that $\mu$ is a probabilistic frame with lower bound $(W_2(\eta, \mu) - \sqrt{A_\eta})^2$. Now let $\nu= {{\bf S}_{\mu}^{-1}}_\# \mu = \frac{1}{N}\sum\limits_{k=1}^N \delta_{{\bf S}_{\mu}^{-1}{\bf f}_k}$ be the canonical dual to $\mu$. Then the upper frame bound for $\nu$ is given by $$C_N = \frac{1}{(W_2(\eta, \mu) - \sqrt{A_\eta})^2} < \frac{1}{(\sqrt{A_\eta}/2)^2} = \frac{1}{A_\eta/4}$$ and $$A_NC_N = \frac{A_N}{(W_2(\eta, \mu) - \sqrt{A_\eta})^2} < \frac{A_\eta/4}{A_\eta/4} = 1,$$ so $\nu$ satisfies the hypotheses of \cref{thm:discrete approx dual}.

By \cref{thm:discrete approx dual}, absolutely continuous probabilistic frames have discrete approximately dual frames. Conversely, discrete probabilistic frames can also have approximate dual frames that are absolutely continuous. For example,  the standard Gaussian measure $d\mu(x) = \frac{1}{\sqrt{2\pi}} e^{-\frac{(x-1)^2}{2}}dx$ is a dual probabilistic frame of $\delta_1$ with respect to the product measure $\mu \otimes \delta_1$. Thus, $\mu$ is also an approximately dual frame to $\delta_1$. This is another example showing that not all the dual frames and approximately dual frames to $\delta_1$ are of pushforward type. 

The above results are inspired by \cref{prop:quadratic closeness}. We can obtain other results if probabilistic frames are perturbed in different ways.

\begin{lemma}
Let $\mu$ be a probabilistic frame and $\nu$ a dual probabilistic frame of $\mu$ with second moment $M_2(\nu)$. Suppose $C>0$, $M_2(\nu)  C < 1$, and there exist $\eta \in \mathcal{P}_2(\mathbb{R}^n)$ and $\gamma \in \Gamma(\eta, \mu)$ such that  for any ${\bf f} \in \mathbb{R}^n$,
    \begin{equation*}
  \int_{\mathbb{R}^n \times \mathbb{R}^n}  |\left\langle {\bf f}, {\bf x-y} \right\rangle|^2 d\gamma({\bf x, y}) \leq C    \Vert  {\bf f} \Vert^2. 
    \end{equation*}
Then $\eta$ is a probabilistic frame and $\nu$ is an approximately dual frame to $\eta$. 
\end{lemma}
\begin{proof}
Since $\nu$ is a dual probabilistic frame to $\mu$, there exists $\pi \in \Gamma(\mu, \nu)$ such that
     \begin{equation*}
      \int_{\mathbb{R}^n \times \mathbb{R}^n} {\bf z}{\bf y}^t d\pi({\bf y, z}) = {\bf Id}.
    \end{equation*}
By \cref{gluinglemma} (Gluing Lemma), there exists $\Tilde{\pi} \in \mathcal{P}(\mathbb{R}^n \times \mathbb{R}^n \times \mathbb{R}^n)$ such that ${\pi_{xy}}_{\#}{\Tilde{\pi}} = \gamma$ and ${\pi_{yz}}_{\#}{\Tilde{\pi}}=\pi$. Now let $\Tilde{\gamma} = {\pi_{xz}}_{\#}{\Tilde{\pi}} \in \Gamma(\eta, \nu)$.
Then, for any ${\bf f} \in \mathbb{R}^n$, 
 \begin{equation*}
 \begin{split}
     \norm{ \int_{\mathbb{R}^n \times \mathbb{R}^n} {\bf z} \langle {\bf x}, {\bf f} \rangle d\Tilde{\gamma}({\bf x, z}) - {\bf f}}^2
      &=  \norm{ \int_{\mathbb{R}^n \times \mathbb{R}^n} {\bf z} \langle {\bf x}, {\bf f} \rangle d\Tilde{\gamma}({\bf x, z}) - \int_{\mathbb{R}^n \times \mathbb{R}^n} {\bf z}\langle {\bf y}, {\bf f} \rangle  d\pi({\bf y, z}) }^2 \\
      &= \norm{\int_{\mathbb{R}^n \times \mathbb{R}^n \times \mathbb{R}^n} {\bf z} \langle {\bf x-y}, {\bf f} \rangle  d\Tilde{\pi}({\bf x, y, z}) }^2 \\
      & \leq \left | \int_{\mathbb{R}^n \times \mathbb{R}^n \times \mathbb{R}^n}  \| {\bf z} \| \ |\langle {\bf x-y}, {\bf f} \rangle |  d\Tilde{\pi}({\bf x, y, z}) \right |^2 \\
      & \leq  \int_{\mathbb{R}^n} \| {\bf z} \|^2   d\nu({\bf z})   \int_{\mathbb{R}^n \times \mathbb{R}^n} |\langle {\bf x-y}, {\bf f} \rangle |^2  d\gamma({\bf x, y}) \\
      & \leq M_2(\nu)  C  \ \Vert  {\bf f} \Vert^2,
 \end{split}
    \end{equation*}
where the second inequality follows from the Cauchy–Schwarz inequality. Then
\begin{equation*}
    \norm{ \int_{\mathbb{R}^n \times \mathbb{R}^n} {\bf x} {\bf z}^t d\Tilde{\gamma}({\bf x, z}) - {\bf Id} } = \norm{\int_{\mathbb{R}^n \times \mathbb{R}^n} {\bf z} {\bf x}^t d\Tilde{\gamma}({\bf x, z}) - {\bf Id} }\leq \sqrt{M_2(\nu)  C } < 1 .
\end{equation*}
Hence, $\nu$ is an approximately dual frame to $\eta$. Therefore, $\eta$ is also a probabilistic frame by \cref{DualIsFrame}. 
\end{proof}

The following diagram shows that given an approximate dual pair $\mu$ and $\nu$, if a probability measure $\eta$ is close to $\mu$, then there exists an approximate dual frame $\xi$ of $\eta$ such that they have the same mixed frame operator as $\mu$ and $\nu$.  

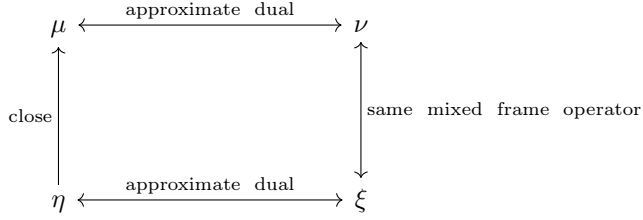
 \begin{center}
  \begin{tikzcd}[column sep=100pt,row sep=50pt]
  \mu \arrow[leftrightarrow, r, "\text{approximate \ dual}"] & \nu \arrow[leftrightarrow, d, "\text{same \ mixed \ frame \ operator}"] \\
 \eta \arrow[u, "\text{close}" ]  \arrow[leftrightarrow, r, "\text{approximate \ dual}"]  & \xi 
 \end{tikzcd}
\end{center}

\begin{lemma}
Let $\mu$ be a probabilistic frame with lower bound $A>0$ and $\nu$ an approximately dual frame to $\mu$.  If there exist $\eta \in \mathcal{P}_2(\mathbb{R}^n)$  and $\gamma \in \Gamma(\mu, \eta) $ such that
    \begin{equation*}
        \int_{\mathbb{R}^n \times \mathbb{R}^n} \| {\bf x}- {\bf w} \|^2 d\gamma({\bf x},{\bf w})  < A,
    \end{equation*}
 then there exists $\xi \in \mathcal{P}_2(\mathbb{R}^n)$  such that $\xi$ and $\eta$ form an approximately dual pair, and they have the same mixed frame operator as $\mu$ and $\nu$.
\end{lemma}
\begin{proof}
By \cref{prop:quadratic closeness}, we know that $\eta$ is a probabilistic frame and thus the frame operator ${\bf S}_\eta$ of $\eta$ is invertible. Since $\mu$ and $\nu$ form an approximate dual pair, then there exists $\gamma' \in \Gamma(\mu, \nu)$ such that 
\begin{equation*}
        {\bf A} := \int_{\mathbb{R}^n \times \mathbb{R}^n} {\bf x} {\bf y}^t d\gamma'({\bf x},{\bf y})
    \end{equation*}
and $\|{\bf A} -{\bf Id}\| <1$.  Now define $T:\mathbb{R}^n \rightarrow \mathbb{R}^n$ by $T({\bf x}) = {\bf A}^t{\bf S}_\eta^{-1}({\bf x})$.
Then, by \cref{PushforwardCharacter} (taking $h({\bf x}) = {\bf 0}$), we know $\xi:=T_\#\eta$ is an approximately dual frame to $\eta$ with respect to $\pi:=({\bf Id}, T)_\#\eta \in \Gamma(\eta, \xi)$. Furthermore, 
$$\int_{\mathbb{R}^n \times \mathbb{R}^n} {\bf w}{\bf z}^t d\pi({\bf w, z}) = \int_{\mathbb{R}^n} {\bf w}T({\bf w})^t d\eta({\bf w}) = {\bf A} = \int_{\mathbb{R}^n \times \mathbb{R}^n} {\bf x} {\bf y}^t d\gamma'({\bf x},{\bf y}).$$
That is to say,  the approximate dual pair $\xi$ and $\eta$ have the same mixed frame operator as $\mu$ and $\nu$.
\end{proof}

If $\gamma \in \Gamma(\mu, \eta) $ is an optimal coupling with respect to $W_2(\mu, \eta)$, we have the following corollary. 
\begin{corollary}
    Let $\mu$ be a probabilistic frame with lower bound $A>0$ and $\nu$ an approximately dual frame to $\mu$.  If there exists $\eta \in \mathcal{P}_2(\mathbb{R}^n)$ such that $W_2(\mu, \eta)  < \sqrt{A}$, then there exists $\xi \in \mathcal{P}_2(\mathbb{R}^n)$  such that $\xi$ and $\eta$ form an approximately dual pair and they have the same mixed frame operator as $\mu$ and $\nu$. 
\end{corollary}

\section{Pseudo-Dual Probabilistic Frames}\label{pseudoDual}

In this section, we generalize dual probabilistic frames even further, to pseudo-dual probabilistic frames. Given $\mu, \nu \in \mathcal{P}_2(\mathbb{R}^n)$ and $\gamma \in \Gamma(\mu, \nu)$, call the matrix  
$$\int_{\mathbb{R}^n \times \mathbb{R}^n} {\bf x}{\bf y}^t d\gamma({\bf x, y})$$
the \emph{mixed frame operator} of $\mu$ and $\nu$. This generalizes the terminology of \Cref{def:approximate dual} slightly, in that neither $\mu$ nor $\nu$ is required to be a probabilistic frame.

\begin{definition}
     Let $\mu$ be a probabilistic frame. A probability measure $\nu \in \mathcal{P}_2(\mathbb{R}^n)$ is called a \emph{pseudo-dual probabilistic frame} of $\mu$ if there exists $\gamma \in \Gamma(\mu, \nu)$ such that the associated mixed frame operator is invertible. 
\end{definition}

Clearly, dual probabilistic frames and approximately dual frames are pseudo-duals. Note that for finite frames $\{{\bf f}_i\}_{i=1}^N$ and $\{{\bf g}_i\}_{i=1}^N$ in $\mathbb{R}^n$ where $N\geq n$, their mixed operator $\sum_{i=1}^N {\bf f}_i {\bf g}_i^t$ is not always invertible. For example, when $d=2$ and $N=4$, define frames
$${\bf f}_1 = \begin{pmatrix} 1 \\ 0 \end{pmatrix}, \quad {\bf f}_2 = \begin{pmatrix} 0 \\ 1 \end{pmatrix}, \quad {\bf f}_3 = \begin{pmatrix} 1 \\ 0 \end{pmatrix}, \quad {\bf f}_4 = \begin{pmatrix} 0 \\ 1 \end{pmatrix}$$
and
$${\bf g}_1 = \begin{pmatrix} 1 \\ 0 \end{pmatrix}, \quad {\bf g}_2 = \begin{pmatrix} 1 \\ 0 \end{pmatrix}, \quad {\bf g}_3 = \begin{pmatrix} 0 \\ 1 \end{pmatrix}, \quad {\bf g}_4 = \begin{pmatrix} 0 \\ 1 \end{pmatrix}.$$
The mixed operator between $\{{\bf f}_i\}_{i=1}^4$ and $\{{\bf g}_i\}_{i=1}^4$ is given by
$\sum\limits_{i=1}^4 {\bf f}_i {\bf g}_i^t = \begin{pmatrix} 1 & 1 \\ 1 & 1 \end{pmatrix}$, which is not invertible. This example shows that a frame and a permutation of that frame need not be pseudo-duals of each other.

Furthermore, arbitrary non-atomic probabilistic frames are not necessarily pseudo-duals. Let $\mu$ be any probabilistic frame on $\R$ with zero mean and non-zero second moment, like the standard Gaussian measure. Then the mixed frame operator of $\mu$ with itself regarding the product measure $\mu \otimes \mu$ is not invertible, since
$$\int_{\mathbb{R} \times \mathbb{R}} xy d\mu(x) d\mu(y) = \left(\int_{\mathbb{R}} x d\mu(x)\right)^2 =0.$$

Similarly to \cref{DualIsFrame}, we claim that a pseudo-dual probabilistic frame is a probabilistic frame. We also obtain an uncertainty principle and a similar relationship between optimal frame bounds for a probabilistic frame and its pseudo-duals. The proof is omitted since the argument is the same as in \cref{DualIsFrame}.

\begin{lemma}
    Let $\mu$ be a probabilistic frame with upper frame bound $B>0$.  Suppose $\nu \in \mathcal{P}_2(\mathbb{R}^n)$ is a pseudo-dual probabilistic frame of $\mu$ with mixed frame operator $\bf A$, then $\nu$ is a probabilistic frame with bounds $ 1/(B \| {\bf A}^{-1}\|^2)$ and $M_2(\nu)$. 
    
In addition, for any ${\bf f} \in \mathbb{R}^n$, we have the following uncertainty principle:
     \begin{equation*}
        \int_{\mathbb{R}^n} |\langle {\bf A }^{-1} {\bf x}, {\bf f} \rangle| ^2 d\mu({\bf x}) \int_{\mathbb{R}^n} |\langle {\bf y}, {\bf f} \rangle| ^2 d\nu({\bf y})  \geq \| {\bf f} \|^4,
    \end{equation*}
and equality is attained for all ${\bf f} \in \mathbb{R}^n$  if and only if  ${\bf S}_{\mu}=c{\bf A}{\bf A}^{t}$ and
$\nu=(\frac{1}{c}{\bf A}^{-1})_{\#}\mu$ where $c$ is a nonzero constant.
Furthermore, if $A_\mu$, $B_\mu$, and $A_\nu$, $B_\nu$ are optimal frame bounds for $\mu$ and $\nu$, then 
\begin{equation*}
    A_\nu \geq \frac{1}{B_\mu \| {\bf A}^{-1}\|^2} \ \text{and} \  A_\mu \geq \frac{1}{B_\nu \| {\bf A}^{-1}\|^2}.
\end{equation*}
\end{lemma}

\begin{lemma}\label{lem:pushpseudo}
Let $\nu \in \mathcal{P}_2(\mathbb{R}^n)$ be a pseudo-dual of the probabilistic frame $\mu$. If the matrix ${\bf A} \in \mathbb{R}^{n \times n}$ is invertible,  then ${\bf A}_\# \nu$ is also a pseudo-dual to $\mu$. If ${\bf A}$ is  the mixed frame operator of $\mu$ and $\nu$, then ${({\bf A}^{-1})^t}_\# \nu$ is a dual probabilistic frame to $\mu$. 
\end{lemma}
    
\begin{proof}
  Clearly, ${\bf A}_\# \nu \in \mathcal{P}_2(\mathbb{R}^n)$. Since $\nu$ is a pseudo-dual frame of $\mu$, then there exists $\gamma \in \Gamma(\mu, \nu)$ such that the mixed frame operator $\int_{\mathbb{R}^n \times \mathbb{R}^n} {\bf x}{\bf y}^t d\gamma({\bf x, y})$ is invertible. Now define $\Tilde{\gamma} := ({\bf Id}, {\bf A})_\# \gamma \in \Gamma(\mu, {\bf A}_\# \nu)$. Then,
    $$\int_{\mathbb{R}^n \times \mathbb{R}^n} {\bf x}{\bf y}^t d\Tilde{\gamma}({\bf x, y}) = \left(\int_{\mathbb{R}^n \times \mathbb{R}^n} {\bf x}{\bf y}^t d\gamma({\bf x, y}) \right) {\bf A}^t,$$
    which is also invertible. Therefore, ${\bf A}_\# \nu$ is a pseudo-dual probabilistic frame to $\mu$. 
    
    If ${\bf A}$ is  the mixed frame operator of $\mu$ and $\nu$, then 
    ${\bf A} = \int_{\mathbb{R}^n \times \mathbb{R}^n} {\bf x}{\bf y}^t d\gamma({\bf x, y})$. Now define $\gamma' := ({\bf Id}, ({\bf A}^{-1})^t)_\# \gamma \in \Gamma(\mu, ({\bf A}^{-1})^t_\# \nu)$. Then
    $$\int_{\mathbb{R}^n \times \mathbb{R}^n} {\bf x}{\bf y}^t d \gamma'({\bf x, y}) =  \left(\int_{\mathbb{R}^n \times \mathbb{R}^n} {\bf x}{\bf y}^t d\gamma({\bf x, y}) \right) {\bf A}^{-1}=  {\bf A} {\bf A}^{-1} = {\bf Id}.$$
    Therefore, $({\bf A}^{-1})^t_\# \nu$ is a dual probabilistic frame to $\mu$ with respect to $\gamma'$. 
\end{proof}

We have the following characterization for pseudo-dual probabilistic frames of pushforward type.
\begin{corollary}\label{cor:psuedodualpush}
    Suppose $\mu$ is a probabilistic frame and $T_\# \mu$ is a pseudo-dual probabilistic frame to $\mu$  where $T: \mathbb{R}^n \rightarrow \mathbb{R}^n$ is measurable. Then for any ${\bf x} \in \mathbb{R}^n$, 
    \begin{equation*}
        T({\bf x}) = {\bf A}^t{\bf S}_\mu^{-1}{\bf x} + h({\bf x}),
    \end{equation*}
    where $\bf A$ is an invertible $n \times n$ matrix and $h: \mathbb{R}^n \rightarrow \mathbb{R}^n$ is such that $h_\#\mu \in \mathcal{P}_2(\mathbb{R}^n)$ and $ \int_{\mathbb{R}^n } {\bf x} h({\bf x})^t d\mu({\bf x}) = {\bf 0}_{n \times n}$. 
\end{corollary}
\begin{proof}
 If $T$ is of the above type, then $T_\# \mu \in \mathcal{P}_2(\mathbb{R}^n)$. Now define $\gamma = ({\bf Id}, T)_\#\mu$. Then 
 $$\int_{\mathbb{R}^n \times \mathbb{R}^n} {\bf x} {\bf y}^t d\gamma({\bf x, y}) = \int_{\mathbb{R}^n } {\bf x} {\bf x}^t d\mu({\bf x}) \ {\bf S}_\mu^{-1} {\bf A} + {\bf 0}_{n \times n} = {\bf A}, $$ 
 which is invertible. Therefore,  $T_\# \mu$ is a pseudo-dual frame. Conversely, if $T_\# \mu$ is a pseudo-dual, then the mixed frame operator ${\bf A}:= \int_{\mathbb{R}^n } {\bf x} T({\bf x})^t d\mu({\bf x}) $ is invertible. Clearly, for any ${\bf x} \in \mathbb{R}^n$, 
    $T({\bf x}) = {\bf A}^t{\bf S}_\mu^{-1}({\bf x}) + T({\bf x}) - {\bf A}^t{\bf S}_\mu^{-1}({\bf x}).$
    Now let $h({\bf x}) = T({\bf x}) - {\bf A}^t{\bf S}_\mu^{-1}({\bf x})$, then $h_\#\mu \in \mathcal{P}_2(\mathbb{R}^n)$ and 
    $$\int_{\mathbb{R}^n } {\bf x} h({\bf x})^t d\mu({\bf x}) =  {\bf A} - \int_{\mathbb{R}^n } {\bf x} {\bf x}^t d\mu({\bf x}) \ {\bf S}_\mu^{-1} {\bf A} = {\bf 0}_{n \times n}.$$
    Thus $T$ is of the desired type.
\end{proof}

\section*{Acknowledgments}
Dongwei used DeepSeek as an internet search engine and thinking tool for the isometric isomorphism proof of \cref{lem:equivalent implies same redundancy}. Dongwei would also like to thank Kasso A. Okoudjou for his hospitality during a visit to Tufts University and for helpful discussions about the redundancy of probabilistic frames and the differences between probabilistic and continuous frames.
This work was partially supported by the National Science Foundation (DMS–2107700; Clayton Shonkwiler).

\bibliographystyle{plain} 
\bibliography{refs}
\end{document}